\documentclass{lms}

\usepackage{amsmath,amsfonts,amssymb,tabularx}
\usepackage{graphicx}
\usepackage{color}
\usepackage{hyperref}
\usepackage{nameref}

\newtheorem{theorem}{Theorem}[section] 
\newtheorem{lemma}[theorem]{Lemma}     

\newtheorem{proposition}[theorem]{Proposition}

\newnumbered{assertion}{Assertion}    
\newnumbered{conjecture}{Conjecture}  
\newnumbered{definition}{Definition}
\newnumbered{hypothesis}{Hypothesis}
\newnumbered{remark}{Remark}
\newnumbered{note}{Note}
\newnumbered{observation}{Observation}
\newnumbered{problem}{Problem}
\newnumbered{question}{Question}
\newnumbered{algorithm}{Algorithm}
\newnumbered{example}{Example}
\newunnumbered{notation}{Notation} 

\newcommand\R{\mathbb{R}}
\newcommand\rr{\mathbb{R}}

\newcommand\RN{\mathbb{R}^N}

\newcommand\La{(-\Delta)^{\alpha/2}}
\newcommand\ul{u_{\lambda}}

\def\ml{\mathcal{L}}
\DeclareMathOperator*{\limess}{lim\,ess}
\newcommand{\sgn}{\,{\rm sgn}}

\title[Fractional diffusion-convection equations]
 {Asymptotic behaviour of solutions to fractional diffusion-convection equations} 

\author{Liviu I. Ignat and Diana Stan}

\classno{35K65,  
26A33, 
35B40, 
35B65, 
76E06}

\extraline{\textbf{Acknowledgments.}   L.I. Ignat was  partially supported by the Project {PN-III-P4-ID-PCE-2016-0035} of the Romanian National Authority for Scientific Research CNCS-UEFISCDI  and by the MINECO project MTM2014-52347, Spain and FA9550-15-1-0027 of AFOSR.
D. Stan was partially supported by the  MEC-Juan de la Cierva postdoctoral fellowship
number FJCI-2015-25797 and by the projects PN-II-RU-TE- 2014-4-0007 of the Romanian National Authority for Scientific Research CNCS--UEFISCDI, by the ERCEA Advanced Grant 2014 669689 - HADE, by the MINECO project {MTM2014-53850-P}, by Basque Government project\\
 IT-641-13 and also by the Basque Government through the BERC 2014-2017 program and by Spanish Ministry of Economy and Competitiveness MINECO: BCAM Severo Ochoa excellence accreditation SEV-2013-0323.
}

\begin{document}
\maketitle

\begin{abstract}
We consider a convection-diffusion model with linear fractional diffusion in the sub-critical range. We prove that the large time asymptotic behavior of the solution is given by the unique entropy solution of the convective part of the   equation.  The proof is based on suitable a-priori estimates, among which proving an Oleinik type inequality plays a key role.
\end{abstract}

\section{Introduction and main results}

We consider the convection diffusion equation
\[
u_{t}(t,x) + (-\Delta)^{\alpha/2}u(t,x)+(f(u(t,x)))_x=0  \quad \text{for }t>0 \text{ and } x \in \mathbb{R},
   \tag{CD}\label{CD}
\]
where $u: (0,\infty)\times \R \to \R$, $(-\Delta)^{\alpha/2}$ is the Fractional Laplacian operator of order $\alpha \in (0,2)$ and $f(\cdot)$ is a locally Lipschitz function whose prototype is $f(s)=|s|^{q-1}s/q$ with $q>1$.
This  model has received  considerable attention since the 1990s due to the interesting phenomena that appear: there is a competition between the effects of the diffusion and convection terms. Depending on the parameters $\alpha$ and $q$, the asymptotic behaviour is given by either the solution of the  diffusion equation:
\[
u_{t}(t,x) + (-\Delta)^{\alpha/2}u(t,x)=0  \quad \text{for }t>0 \text{ and } x \in \mathbb{R},
   \tag{D}\label{D}
\]
 or the convective one
\[
u_{t}(t,x) +(f(u(t,x)))_x=0  \quad \text{for }t>0 \text{ and } x \in \mathbb{R},
   \tag{C}\label{C}
\]
  or by a self-similar solution of (CD) in a critical case. The classical case $\alpha=2$ has been analysed for all $q>1$ in the quoted papers of Escobedo, V\'azquez and Zuazua \cite{EVZArma,EVZIndiana,EZ}.

In the last twenty years there has been a great interest in models with nonlocal diffusion, specially fractional diffusion since the fractional Laplacian  $(-\Delta)^{\alpha/2}$ is the infinitesimal generator of a stable Levy process. There are many applications in physical sciences where models with anomalous diffusion are needed, see the survey \cite{WoyczynskiLevyProc2001} for a description of possible applications, and the lecture notes \cite{VazCIME} for a presentation of recent models involving nonlocal diffusion.

We are interested in the large time asymptotic behavior of solutions to the initial value problem
\begin{equation}\label{Problem1}
  \left\{ \begin{array}{ll}
  u_{t}(t,x) + (-\Delta)^{\alpha/2}u(t,x)+(f(u(t,x)))_x=0 &\text{for }t>0 \text{ and } x \in \mathbb{R}, \\[10pt]
  u(0,x)  =u_0(x) &\text{for } x \in \mathbb{R}.
    \end{array}
    \right.
    \end{equation}
The critical case $q=\alpha$ makes the difference in the asymptotic behavior since equation \eqref{CD} is invariant by scaling $u_\lambda(t,x)=\lambda u(\lambda^\alpha t,\lambda x)$, and it admits self-similar solutions. In this case the asymptotic behavior of the solutions is given by the self-similar solution with the same mass as the initial datum $u_0$ (see \cite{BilerFunaki}).
In the supercritical range $\alpha \in (1,2)$, $q> \max\{1,\alpha\}$   the asymptotic behaviour is given by the fundamental  solution of the diffusion model (D)   multiplied by the mass of the initial datum (see \cite{BilerKarchWoyczAsymp} for $\alpha\in(1,2)$).
We will provide more details in next section.

In this paper we consider the case $\alpha \in (1,2)$ and  the nonlinearity $f(u)=|u|^{q-1}u/q$ in the  subcritical range $1<q<\alpha$, which has been an open issue so far. The main result of this paper is the following theorem.

\begin{theorem}\label{ThmAsympBehav}
		For any $1<q<\alpha<2$, $f(u)=|u|^{q-1}u/q$ and $u_0\in L^1(\R)\cap L^\infty(\R)$ nonnegative there exists a  unique mild solution $u \in C([0,\infty), L^1(\R))\cap  L^\infty((0,\infty)\times\R)$ of system \eqref{Problem1}. Moreover, for any $1\leq p<\infty$, solution $u$ satisfies
		\begin{equation}
\label{limit.asymp}
  \lim _{t\rightarrow \infty} t^{\frac 1q(1-\frac 1p)}\|u(t)-U_M(t)\|_{L^p(\R)}=0,
\end{equation}
where $M$ is the mass of the initial data $u_0$ and $U_M$ is the unique entropy solution of the equation
\begin{equation}\label{limit.problem}
  \left\{ \begin{array}{ll}
  u_{t}+(f(u))_x=0 &\text{for }t>0 \text{ and } x \in \mathbb{R}, \\[10pt]
  u(0)  =M\delta_0. &
    \end{array}
    \right.
\end{equation}
\end{theorem}

\begin{remark}
	We believe that the $L^\infty$-assumption on the initial data can be dropped. Through the paper we will consider nonnegative solutions. The general case of
 changing sign solutions can be analysed following the same arguments as in \cite[Section 6]{CazacuIgnatPazoto}.
We emphasise that since the nonlinearity should be locally Lipschitz we should impose $q>1$. Since we are interested in the subcritical case where  the convection is dominant we have to impose $\alpha>q$ and hence $\alpha$ should belong to the interval $(1,2)$.
\end{remark}

An interesting phenomenon happens: the diffusion is dominant over the convection for $\alpha>1$, having a regularizing effect on the solution. However, when $1<q<\alpha$ in the asymptotic limit as time $t$ goes to infinity the solution approaches the unique entropy solution to the pure convective equation which is discontinuous and develops shocks. This phenomenon has been established for the local case $\alpha=2$ by Escobedo, V\'azquez and Zuazua in \cite{EVZArma}. In this paper we prove that this behavior  holds as long as  $1<q<\alpha<2$.  This is done using both parabolic and hyperbolic arguments and dealing with the difficulties created by the nonlocal operator and the nonlinearity of the convective term.

The organization of the paper is as follows. In Section \ref{SectPrelim} we give a panorama on previous results on the model both in local and nonlocal cases. Also we provide a reminder on the diffusion equation which will be useful throughout the paper.
 In Section \ref{SectionExist} we are concerned with the existence and main properties of solutions. Entropy and mild solutions are introduced. The key estimate is given in Proposition \ref{prop.oleinik} where we show that for any $\alpha,q\in (1,2]$ and any initial data uniformly bounded above and below by two positive constants, the solution of our problem satisfies an Oleinik type inequality, $(u^{q-1})_x\leq 1/t$. We emphasize that this estimate does not require $q<\alpha$.
In Section \ref{SectAsymp} we prove the asymptotic behavior of solutions stated in Theorem \ref{ThmAsympBehav}.

\section{Preliminaries}\label{SectPrelim}

\subsection{Panorama: from local to nonlocal diffusion}

 We describe some of the results known so far for this convection-diffusion model. We try to cover all the ranges of parameters and finally to better place our contribution in this field.

The general model is
  \begin{equation}\label{GenEq}
  \left\{ \begin{array}{ll}
 u_t(t,x) + \mathcal{L}[u](t,x)+\overline{b}\cdot \nabla (f(u(t,x)))=0 &\text{for }t>0 \text{ and } x \in \RN, \\[10pt]
  u(0,x)  =u_0(x) &\text{for } x \in \RN,
    \end{array}
    \right.
\end{equation}
  where $\mathcal{L}$ is a L\'evy type operator, $\widehat{\mathcal{L}v}(\xi)=a(\xi)\hat v(\xi)$, whose symbol $a$ is written in the form
  \[
  a(\xi)=i k \xi +\mu(\xi)+\int _{\rr^N}\big(1-e^{-i\eta \xi}-i\eta \xi \mathbf{1}_{|\eta|<1}\big)\Pi (d\eta).
\]
Usually $k\in \rr^N$, $\mu$ is a positive semi-definite quadratic form on $\rr^N$ and  $\Pi$   is a positive Radon measure satisfying
\[
  \int_{\rr^N} \min\{|z|^2, 1\} \Pi(dz)<\infty.
\]
Two particular cases are the Laplacian, $\ml=-\Delta$ and $\ml=(-\Delta)^{\alpha/2}$ corresponding to $k=0$, $\mu(\xi)=|\xi|^2$, $\Pi =0$
and $k=0$, $\mu(\xi)=0$, $\Pi(dz)=|z|^{-N-\alpha}dz$ respectively.

\medskip

\noindent\textbf{Local Diffusion.} The local diffusion case, i.e. $\ml=-\Delta$, has been intensively studied for linear diffusion $u_t - \Delta u + \overline{b}\cdot \nabla (|u|^{q-1}u)=0$, see \cite{EZ} for the supercritical and critical cases ($q\geq 1+1/N$ in $\rr^N$) and \cite{EVZArma} for the subcritical case $1<q<2$ in dimension $N=1$. The subcritical case $q<1+1/N$ in any dimension $N\geq 1$ has been analysed in \cite{EVZIndiana} for nonnegative solutions and  for changing sign solutions in \cite{Carpio1996}.

\medskip

\noindent\textbf{Nonlocal Diffusion.} There is always a competition between the diffusion, which is differentiable of order $\alpha$, and the convection terms having one derivative. This implies the consideration of certain classes of solutions: entropy solutions, weak solutions, mild solutions.  The study takes into consideration the fractional order $\alpha$, the nonlinearity $f(u)$, the dimension $N$ and the regularity of the initial data $u_0$.

\noindent\textbf{Existence of solutions.} For all ranges or parameters $\alpha \in (0,2)$, $q>1$,  the model admits a unique entropy solution. More precisely,
 for $\alpha \in (1,2)$ and $f$ locally Lipshitz,  the existence and uniqueness of entropy solutions were proved by Droniou \cite{DroniouVanishing2003}.
Then  Alibaud \cite{AlibaudEntropy} proved the same for $\alpha \in (0,2)$.
 Cifani and Jakobsen \cite{CifaniJakobsenEntropySol}  proved the existence of entropy solutions for the degenerate nonlinear nonlocal integral equation $u_t+(-\Delta)^{\alpha/2}A(u)+(f(u))_x=0$ with  $\alpha \in (0,2) $ and developed a numerical scheme that gives an idea of the asymptotic behavior of the solution.

The existence of entropy solutions for \eqref{GenEq} with merely bounded (possibly non-integrable) data has been proved by Endal and Jakobsen \cite{EndalJakobsen}.  If moreover $f\in C^\infty$, $\alpha \in (1,2)$  and $q>1$  then there exists a unique mild solution with good regularity properties, see Droniou, Gallouet, Vovelle \cite{Droniou}.

When the diffusion is smaller, $\alpha\in (0,1]$ regularity is lost, since the convection has the effect of shock formation. There is non-uniqueness of weak solutions, as proved by Alibaud and Andreianov \cite{AlibaudAndreianov}. However, uniqueness holds in the class of entropy solutions.

\smallskip

\noindent\textbf{Asymptotic Behaviour.} Concerning the asymptotic behavior of solutions there are previous works in some ranges of exponents.

\emph{(i) Integrable data.} When the data is $u_0 \in L^1(\RN)$ there are previous works in the critical and supercritical cases. The critical case corresponds to $q=1+\frac{\alpha-1}{N}$ when the equation \eqref{CD} admits a unique self-similar solution $U(t,x)=t^{-N/\alpha}U(1,xt^{-1/\alpha})$ with data $U(0,x)=M\delta(x).$
For $\alpha\in (1,2)$ the critical case has been analyzed Biler, Karch and  Woyczy{\'n}ski \cite{BilerKarchWoycz2001} who proved that
the asymptotic profile as $t\to \infty$ is given by the self-similar solution $ U(t,x)$ described above.
When $\alpha\in (0,1)$ the critical exponent $q$ is less than one and the nonlinearity would  not be Lipschitz which is out the scope of this analysis.

In the supercritical case  $q>1+(\alpha-1)/{N}$, $\alpha\in (1,2)$, the diffusion is dominant and then the asymptotic behavior of solutions to \eqref{Problem1} with $u_0\in L^1(\rr)\cap L^\infty(\rr)$ is given by $e^{-t(-\Delta)^{\alpha/2}}u_0$, the solution  of the linear diffusion problem $U_t + (-\Delta)^{\alpha/2} U=0$ with data $U(0,x)=u_0(x)$ (see Biler, Karch and  Woyczy{\'n}ski \cite[Th.~4.1, Lemma~4.1]{BilerKarchWoyczAsymp}). Some results in the one dimensional case were obtained by Biler, Funaki and Woyczy{\'n}ski \cite{BilerFunaki}. The analysis of the linear semigroup generated by \eqref{D} shows that the first term in the asymptotic behaviour may be chosen as $M K^\alpha_t$ where $K^\alpha_t$ is the fundamental solution of problem \eqref{D}. See for instance \cite[Theorem 6.3]{BSVfracHE}.
In Section \ref{reminder} we present more details about the linear model \eqref{D} and its properties.

When $\alpha\in (0,1)$ all the nonlinearities considered here are super-critical since $q>1>1+(\alpha-1)/N$. The  asymptotic behavior is given again by the linear semigroup. We state in the following theorem the result in the one-dimensional case.
\begin{theorem}
For any  $\alpha \in (0,1)$,  $q>1$, $f(u)=|u|^{q-1}u/q$ and $u_0\in L^1(\R)\cap L^\infty(\R)$  there exists a  unique entropy solution $u$ of system \eqref{Problem1}. Moreover, for any $1\leq p<\infty$, solution $u$ satisfies
		\begin{equation*}
  \lim _{t\rightarrow \infty} t^{\frac 1 \alpha(1-\frac 1p)}\|u(t)-U(t)\|_{L^p(\R)}=0,
\end{equation*}
where $U$ is the unique weak solution of the equation
\begin{equation*}
  \left\{ \begin{array}{ll}
 U_{t}(t,x)+(-\Delta)^{\alpha/2}U(t,x)=0 &\text{for }t>0 \text{ and } x \in \mathbb{R}, \\[10pt]
  U(0,x)  =u_0(x)  & \text{for } x \in \mathbb{R}.
    \end{array}
    \right.
\end{equation*}
\end{theorem}
\begin{proof}
The proof should follow as in \cite[Th.~1.1, Th.~3.5]{AlibaudImbertKarch} by using the technique of approximation with a vanishing viscosity term:
$$
(u_ \epsilon)_t +(-\Delta)^{\alpha/2} u_\epsilon+(f(u_\epsilon))_x=  \epsilon \Delta u_\epsilon.
$$
The asymptotic behavior is proved first for this approximating problem and then by letting $\epsilon\rightarrow 0$ for the initial problem.
We could also  work directly with entropy solutions as in this present paper, but one should  consider a parabolic scaling 
$u_\lambda(t,x)=\lambda^2 u(\lambda^2 t,\lambda x)$ instead of the one used in Section \ref{SectAsymp}.
A detailed proof of these fact does not bring great novelty and we consider it is beyond the purpose of this paper.
\end{proof}

 In this work we make a step further by describing the asymptotic behavior of mild solutions in the subcritical case $1<q<1+(\alpha-1)/{N}$ and dimension one, that is $1<q<\alpha<2$, for bounded integrable data.

    \medskip
\emph{(ii) Step-like data.} There is an interesting phenomenon when $f(u)=u^2/2$ supplemented by a step-like initial datum  approaching the constants $u_{\pm}$, $u_-<u_+$, as $x\rightarrow \pm \infty$, respectively. For $\alpha\in (1,2)$ in \cite{KarchMiaoXu} the authors  study the one dimensional case and they prove that the limit profile is given by a rarefaction wave, that is the unique entropy solution of the Riemann problem
 \[
  w_t+ww_x=0, \quad
  w(0,x)=
  \left\{
  \begin{array}{ll}
   u_-, &x<0, \\
  u_+,&x>0.
    \end{array}
  \right.
\]
When $\alpha \in (0,1)$ the convection is negligible and the asymptotic behavior is given by the solution of the diffusion problem \eqref{D} with the same initial initial data $w(0,x)$ as above.
This is proved in \cite{AlibaudImbertKarch}  in dimension one.
The two-dimensional case of the above results has been analysed by Karch, Pudelko and Xu \cite{KarchPudelkoXu}. The characterization depends on the fractional order $\alpha$ and on the direction $\overline{b}$ of the convective nonlinearity in \eqref{GenEq}.

\medskip

\noindent\textbf{Remarks. } (i) There is a connection with Hamilton-Jacobi equations. By considering the integrated solution $v(t,x)=\int_{-\infty}^x u(t,y)dy$, it follows that $v(t,x)$ solves the equation $v_t + \La[v] + \frac{1}{q}(v_x)^q=0$, which is a type of Hamilton-Jacobi equation with fractional diffusion. The problem admits classical solutions when $\alpha \in (1,2)$ (\cite{DroniouImbertFractal,Imbert2005218}). For $\alpha=1$ this is related to drift-diffusion equations (\cite{Silvestre20112020}).

 (ii) There is a considerable interest in nonlocal equations with zero-order operators $\mathcal{L}[u]= J \star u - u$, where $J$ is a non-singular, integrable kernel with mass one. This is a quite different topic, since the nonlocal operator does not provide any regularity for the solution, as it happens in the fractional derivative case, and then  other techniques must be used. When $q<2$, the first author considers the model $u_t = J \star u - u - (f (u))_x$ in \cite{CazacuIgnatPazoto}. The asymptotic behavior is given by the solution of \eqref{limit.problem}. The case $q=2$ has been analyzed in \cite{MR2138795} and $q>2$ in \cite{MR3190994}. There are situations when the convection is also nonlocal, $u_t = J \star u - u + G\ast f(u)-f(u)$. We refer to \cite{MR2356418} for the supercritical case $q>1+1/N$ and \cite{MR3328145} for the critical case $q=1+1/N$. However, for the subcritical case, i.e. $q<1+1/N$ there are no results on the long time behavior of the solutions.

(iii) The case of nonlinear local diffusion also brings considerable difficulties, for instance for porous-medium type diffusion and convection the model becomes $u_t=\Delta u^m- (u^q)_x$.  The third parameter $m$ of the nonlinearity changes the behaviour of the solution.  For slow diffusion and slow convection we refer to Lauren\c{c}ot and Simondon \cite{LaurencotSimondon}. See \cite{LaurencotFast} for fast convection $0<q<1$ and slow diffusion $m>1$. The asymptotics of both fractional and nonlinear diffusion, $(-\Delta)^{\alpha/2}(u^m)$ plus convection has not been considered as far as we know.

\subsection{Reminder on linear fractional diffusion}\label{reminder}
We recall some useful results concerning  the associated diffusion problem \eqref{D}, that is the \emph{Fractional Heat Equation}  for $0<\alpha<2$. We consider the initial value problem
\begin{equation}\label{FHE}
  \left\{ \begin{array}{ll}
U_t (t,x) + (-\Delta)^{\alpha/2} U(t,x)=0 \quad \text{for } x \in \mathbb{R} \text{ and }t>0,\\
   U(0,x)=U_0(x) \quad \text{for } x \in \mathbb{R}.
    \end{array}
    \right.
\end{equation}
    This problem has been widely studied and many results are known  (see \cite{Applebaum,Bertoin,BlumenthalGetoor} for the probabilistic point of view,  \cite{Valdinoc} for a nice motivation of the model and the recent survey \cite{BSVfracHE} for a complete characterization). Some useful properties  are proved in \cite[Section 2]{Droniou}.
    For initial  data $U_0 \in L^1(\R)$ the solution of Problem  \eqref{FHE} has the integral representation
\begin{equation*}
U(t,x)=(K^{\alpha}_t(\cdot) \star U_0 ) (x)= \int_{\R}K_t^{\alpha}(x-z)U_0(z)dz\,,
\end{equation*}
where the kernel $K_t^\alpha$ has Fourier transform $\widehat{K}_t^{\alpha}(\xi)=e^{-|\xi|^{\alpha}t}.$
If $\alpha=2$, the function $K^2_t$ is the Gaussian heat kernel. We recall some detailed information on the behaviour of the kernel $K_t^{\alpha}(x)$ for $0<\alpha<2$. In the particular case $\alpha=1$, the kernel is explicit, given by the
 formula
 $$
 K^{1}_t(x)=C t (|x|^2+t^2)^{-1}.
 $$
Kernel $K_t^{\alpha}(x)$ is the fundamental solution of Problem \eqref{FHE}, that is $K_t^{\alpha}(x)$ solves the problem with initial data Dirac delta
 $
\delta_0.
 $
 It is known \cite{BlumenthalGetoor} that the kernel $K_t^{\alpha}$ has the self-similar form
 $$K_t^{\alpha}(x)=t^{-1/\alpha}F_\alpha(|x|t^{-1/\alpha}),$$
 for some profile function, $F_\alpha(r)$. For any $\alpha\in (0,2)$ the profile $F_\alpha$ is $C^\infty(\rr)$, positive and decreasing on $(0,\infty)$,  and behaves at infinity like $F_\alpha(r) \sim r^{-(1+\alpha)}$. Moreover, the solution of Problem \eqref{FHE} behaves as time $t\to \infty$ as  $M K_t^{\alpha}$, where $M=\int_{\R}U_0(x) dx$ is the total mass:
\begin{equation*} 
t^{  \frac{1}{\alpha}\left( 1- \frac{1}{ p} \right)} \|U(t,\cdot) -M K_t^{\alpha}(\cdot)\|_{L^{p}(\rr)} \to 0   \quad\text{as}\quad  t\to \infty.
\end{equation*}
See for instance \cite[Theorem 6.3]{BSVfracHE}. Throughout the paper we will need the following time decay estimates on the fractional derivatives of the kernel.

 \begin{lemma}
 	\label{decay.nucleu}
 	For any $\alpha\in (0,2)$, $s\geq 0$ and $1\leq p\leq \infty$ the kernel $K_t^\alpha$ satisfies the following estimates for any positive $t$:
\begin{eqnarray}
\|   K_t^{\alpha} \|_{L^{p}(\R)} &\simeq & \mathcal{K}    t^{-\frac{1}{\alpha}(1-\frac 1p)},  \label{EstFHE}\\
\||D|^s K^\alpha_t\|_{L^p(\R)}& \lesssim & t^{-\frac {1}\alpha (1-\frac 1p)-\frac s\alpha},\label{EstFHE2}\\
\||D|^s \partial_x K^\alpha_t\|_{L^p(\R)}& \lesssim & t^{-\frac {1}\alpha (1-\frac 1p)-\frac {s+1}\alpha}\label{EstFHE3}.
\end{eqnarray}
\end{lemma}
\noindent We used the notation $|D|^s:=(-\Delta)^{s/2}$. The proof of these estimates is given in the Appendix.

\section{Existence of solutions and main properties}\label{SectionExist}

\subsection{Concept of solution: entropy and mild solutions}

We now recall some classical results for systems  \eqref{Problem1}  and \eqref{limit.problem}. In the case of the conservation law \eqref{limit.problem} the entropy formulation is as follows.

 \begin{definition}\label{entropysol} (\cite{MR735207})
By an entropy solution of system \eqref{limit.problem}  we mean a function
\[
w\in L^\infty((0,\infty),L^1(\R))\cap L^\infty((\tau,\infty)\times \R), \ \forall \tau\in (0,\infty)
\]
such that:

 C1) For every constant $k\in \R$ and $\varphi\in C_c^\infty((0,\infty)\times \R)$, $\varphi\geq 0$,  the following inequality holds
\[
\int_0^\infty \int_{\R} \Big(|w-k|\frac{\partial \varphi}{\partial t} +\sgn(w-k)(f(w)-f(k))\frac{\partial \varphi}{\partial x}\Big) dxdt\geq 0.
\]

C2) For any bounded continuous function $\psi$
\begin{equation*}
\limess_{t\downarrow 0} \int _\R w(t,x)\psi(x)dx=M\psi(0).
\end{equation*}
\end{definition}
The existence of a unique entropy solution of system \eqref{limit.problem}, as well as its properties were deeply analysed in \cite{MR735207}.
For $f(u)=|u|^{q-1}u/q$ system \eqref{limit.problem} has an unique entropy solution $U_M$, see  \cite[Section 2]{MR735207},
which is given by the $N$-wave profile
\begin{equation*} 
U_M(t,x)=\left\{
\begin{array}{ll}
(x/t)^{\frac 1{q-1}},  & 0<x<r(t),     \\[10pt]
  0, & otherwise,
\end{array}
\right.
\end{equation*}
with $r(t)=(\frac q{q-1})^{\frac{q-1}q}M^{(q-1)/q}t^{1/q}$.

Let us first recall the  representation of the fractional Laplacian in \cite{DroniouImbertFractal}. For any $\alpha\in (0,2)$: there exists  a positive constant $c(\alpha)$ such that for all $\varphi\in \mathcal{C}^2_b(\R)$, all $r>0$ and all $x\in \R$ the following holds
\begin{equation}
\label{frac.lap}
  [(-\Delta)^{\alpha/2}\varphi ](x)
=-c(\alpha)\int _{|z|\geq r}\frac{\varphi(x+z)-\varphi(x)}{|z|^{1+\alpha}}dz-c(\alpha)\int _{|z|\leq r}
\frac{\varphi(x+z)-\varphi(x)- \varphi'(x)z}{|z|^{1+\alpha}}dz.
\end{equation}
Using this representation, we introduce, according to \cite{AlibaudEntropy}, the following definition of the entropy solution for system \eqref{Problem1}.
\begin{definition}\label{DefnEntropySol} (\cite{AlibaudEntropy})
Let $u_0\in L^\infty(\R)$. We define an entropy solution of Problem \eqref{Problem1} as a function $u\in L^\infty((0,\infty)\times \R)$  such that for all $r>0$, all non-negative $\varphi\in C_c^\infty ([0,\infty)\times \R)$, all smooth convex functions $\eta:\R\rightarrow\R $ and all $\phi$ such that $\phi'=\eta' f'$, $f(s)=|s|^{q-1}s/q$,
\begin{align*}
  \int _0^\infty\int_{\R} &(\eta(u)\partial_t \varphi +\phi(u)\partial _x\varphi )dx dt\\
\nonumber  &+c(\alpha)  \int _0^\infty\int_{\R} \int _{|z|\geq r}\eta '(u(t,x)) 
\frac{u(t,x+z)-u(t,x)}{|z|^{1+\alpha}} \varphi(t,x) dzdxdt+\\
\nonumber &+c(\alpha)  \int _0^\infty\int_{\R} \int _{|z|\leq r} \eta (u(t,x))
\frac{\varphi(t,x+z)-\varphi(t,x)- \varphi'(t,x)z}{|z|^{1+\alpha}}dzdxdt\\
\nonumber &+ \int_{\R} \eta(u_0)\varphi (0,x)dx\geq 0.
\end{align*}
\end{definition}

\begin{remark}
In the above definition it is sufficient to consider the particular entropy-flux pairs, $\eta _k(s)=|s-k|$, $\varphi_k(s)=\sgn(s-k)(f(s)-f(k))$, for any real number $k$.
\end{remark}

For any $u_0\in L^\infty(\R)$ and $f:\rr\rightarrow\rr$ locally Lipschitz
 there exists a unique entropy solution of Problem \eqref{Problem1}. Entropy solutions belong to $C([0,\infty),L^1_{loc}(\R))$. If $u_0\in L^1(\R)\cap L^\infty(\R)$, then so does $u(t)$, for all $t>0$, and moreover  $u\in C([0,\infty),L^1(\R))$.
 All these properties have been proved in \cite{Droniou,AlibaudEntropy}.
 In the above papers the authors introduce a splitting in time approximation in order to prove the existence of an entropy solution. In fact for any $\delta>0$ they define the approximation $u_\delta$ in the following way: let $u^\delta(0,\cdot)=u_0$; for all $n\geq 0$, on the time interval $(2n\delta,(2n+1)\delta]$, $u^\delta$ is the solution of $\partial_tu^\delta +2(-\Delta)^{\alpha/2}u^\delta=0$ with initial condition $u^\delta(2n\delta,\cdot)$, and on the time interval $((2n+1)\delta,2(n+1)\delta]$, $u^\delta$ is the entropy solution of $\partial_tu^\delta +2\partial_x(f(u^\delta))=0$ with initial condition $u^\delta((2n+1)\delta,\cdot)$. For any initial data  in $L^\infty(\rr)$ the approximation $u^\delta$ converges in $C([0,T),L^1_{loc}(\rr))$, $T>0$, to the entropy solution of Problem \eqref{Problem1}.

In \cite{Droniou}, for $\alpha\in (1,2)$, and \cite{AlibaudEntropy} for $0<\alpha<1$, the authors prove that the entropy solutions in the sense of Definition \ref{DefnEntropySol} are solutions in the sense of distributions. Moreover when $\alpha\in (1,2)$, Droniou \cite{Droniou} proved that this distributional  solution is the unique mild solution in the sense of Definition \ref{mild} below.

\begin{definition}\label{mild}Let $u_0\in L^\infty(\rr)$ and $T>0$ or $T=\infty$.
We say that a \emph{mild solution} of Problem \eqref{Problem1} is a function $u\in L^\infty((0,\infty)\times \rr)$ which satisfies for a.e. $(t,x)\in (0,T)\times \rr$,
\begin{equation}
\label{mild.form}
  u(t,x)=(K_t^{\alpha} \star u_0)(x)  + \int_0^t (K_{t-\sigma}^{\alpha})_x \star  f(u(\sigma,x)) d\sigma.
\end{equation}
\end{definition}

The existence and regularity of the mild solution are given in the following Proposition.

\begin{proposition}   \label{PropExistence}
For any  $u_0\in L^{\infty}(\R)$  there exists a unique global mild solution $u$ of Problem \eqref{Problem1}. Moreover $u$ satisfies:

(i) $\operatorname*{ess\,inf} u_0 \le u(t,x) \le \operatorname*{ess\,sup} u_0$.

If $u_0 \in L^1(\R) \cap L^{\infty}(\R)$ then

(ii)  $u \in C([0,+\infty), L^1(\R))\cap C((0,\infty), L^\infty(\R))$. Moreover,   $\|u(t)\|_{L^1(\R)} \le  \|u_0\|_{L^1(\R)}$.

(iii) for any $s<\alpha+\min\{\alpha,q\}-1$ and $1<p<\infty$  solution  $u$ satisfies $u_t \in C((0,\infty) ,L^p( \R))$ and $u \in C((0,\infty),H^{s,p}(\R))$.
\end{proposition}

\begin{remark}
	Since $\alpha+\min\{\alpha,q\}-1>1$ we have for any $t>0$ that $u_x(t)\in L^p(\rr)$ for any $1<p<\infty$. Moreover for any $t>0$, the map $x\mapsto u(t,x)$ is continuous. The last property also guarantees that various integrations by parts used in the paper are allowed.
\end{remark}

\begin{proof}The global existence, uniqueness and the first two properties are proved in \cite{Droniou}.
We now prove property (iii).
 Its proof relays on a classical bootstrap argument: one starts with some regularity of $u$ in the right-hand side and obtain that this right hand side term is slightly better than the hypothesis. For a nice review of the method we refer to \cite[Ch.~1.3, p.~20]{MR2233925}.
Let us fix $T>0$.	We first remark that since $u\in C([0,T],L^1(\R))\cap L^\infty((0,\infty)\times \R)$ we have that $f(u)=|u|^{q-1}u/q$ belongs to the same space. Moreover, it is sufficient to prove that for any $t>0$, $u(t)\in H^{s,p}(\rr)$ with a norm that is  bounded in any interval $[\tau,T]$ with $\tau>0$.

The main steps of the proof are as follows: we first prove that for $u\in L^\infty((0,T),L^1(\rr)\cap L^\infty(\rr))$ the right hand side in \eqref{mild.form} belongs to $H^{s,p}(\rr)$ for any $0<s<\alpha-1$, $1<p<\infty$. The next step is to use this new regularity to prove the same for $0<s<\alpha$. The last step, the most technical one, is to extend the regularity up to $s<\alpha+\min\{\alpha,q\}-1$.

\textit{Step I.} We first prove that  we gain some regularity for  $u$, $u\in C((0,T),H^{s,p}(\R))$ for any $0<s<\alpha-1$ and $1<p<\infty$. Let  $0<s<\alpha-1$. We have
\begin{equation}\label{identity.s}
  |D|^su(t)=|D|^sK^\alpha_t\ast u_0 +\int _0^t |D|^s\partial _xK^\alpha_{t-\sigma}\ast f(u(\sigma))d\sigma.
\end{equation}
Using  the decay of the $s$ derivative of $K_t^\alpha$ in \eqref{EstFHE2}, \eqref{EstFHE3}
and that $0<s<\alpha-1$ we find that for any $1<p<\infty$ the following holds for any $t\in (0,T]$:
\begin{align*}
 \| |D|^su(t) \|_{L^p(\R)}&\leq \| |D|^sK^\alpha_t\|_{L^1(\R)}\| u_0 \|_{L^p(\R)}+\int _0^t
 \||D|^{s} \partial_xK^\alpha_{t-\sigma}\|_{L^1(\R)} \|f(u(\sigma))\|_{L^p(\R)}d\sigma\\
 &\lesssim  t^{-\frac s\alpha}+\int _0^t (t-\sigma)^{-\frac{s+1}\alpha}d\sigma=
  t^{-\frac s\alpha}(1+t^{1-\frac 1\alpha})\lesssim   t^{-\frac s\alpha}.
\end{align*}

Let us now explain why identity \eqref{identity.s} holds. We know that $u_0\in L^1
(\rr)\cap L^\infty(\rr)$ and by Lemma \ref{decay.nucleu} kernel $K^\alpha_t $ satisfies $|D|^sK^\alpha_t\in L^p(\rr)$ for any $1\leq p\leq \infty$. Hence $|D|^s (K^\alpha_t\ast u_0)=(|D|^sK^\alpha_t)\ast u_0$. Let us now prove that for a.e. $x\in \rr$ the following holds
\begin{equation}\label{change.diff.int}
  |D|^s \int _0^t \partial _xK^\alpha_{t-\sigma}\ast f(u(\sigma))d\sigma= \int _0^t |D|^s\partial _xK^\alpha_{t-\sigma}\ast f(u(\sigma))d\sigma.
  \end{equation}
  For any $\rho>0$, the Tonelli-Fubini theorem can be applied to obtain that
\begin{equation}\label{eq.221}
  |D|^s \int _0^{t-\rho} \partial _xK^\alpha_{t-\sigma}\ast  f(u(\sigma))d\sigma= \int _0^{t-\rho} |D|^s\partial _xK^\alpha_{t-\sigma}\ast f(u(\sigma))d\sigma.
\end{equation}
Indeed, \eqref{eq.221} is true since we avoid the singularity of $K^\alpha_{t}$ at $t=0$. Moreover, as $\rho\rightarrow 0$,  $0<s<\alpha-1$, using \eqref{EstFHE3} we obtain that for any $1\leq p\leq \infty$ the following holds
\begin{align*}
\int_{t-\rho}^t &\| |D|^s \partial _xK^\alpha_{t-\sigma}\ast  f(u(\sigma))\|_{L^p(\rr)} d\sigma  \le \int_{t-\rho}^t \||D|^s \partial _xK^\alpha_{t-\sigma}\|_{L^1(\R)}  \|  f(u(\sigma))\|_ {L^p(\R) } d \sigma \\
 &\lesssim \int _{t-\rho}^t (t-\sigma)^{-\frac{s+1}\alpha}d\sigma= \rho^{1-\frac{s+1}\alpha}  \to 0.
\end{align*}
Similarly, using \eqref{EstFHE2} it follows that
$$\int_{t-\rho}^t \| \partial _xK^\alpha_{t-\sigma}\ast  f(u(\sigma))\|_{L^p(\rr)} d\sigma    \to 0 \quad \text{as } \rho \to 0.$$
Therefore we obtain that 
\begin{equation*}
\int _0^{t-\rho} \partial _xK^\alpha_{t-\sigma}\ast  f(u(\sigma))d\sigma\rightarrow \int _0^{t} \partial _xK^\alpha_{t-\sigma}\ast  f(u(\sigma))d\sigma
\end{equation*}
and 
\begin{equation}\label{eq.222}
\int _0^{t-\rho} |D|^s\partial _xK^\alpha_{t-\sigma}\ast  f(u(\sigma))d\sigma\rightarrow \int _0^{t} |D|^s\partial _xK^\alpha_{t-\sigma}\ast  f(u(\sigma))d\sigma
\end{equation}
in any $L^p(\rr)$, $1\leq p\leq \infty$. In view of \eqref{eq.221} and \eqref{eq.222} we obtain that $|D|^s\int _0^t \partial _xK^\alpha_{t-\sigma}\ast f(u(\sigma))d\sigma$ belongs to $L^p(\rr)$ for any $1\leq p\leq\infty$ and moreover \eqref{change.diff.int} holds in $L^p(\rr)$, $1\leq p\leq\infty$, so for a.e. $x\in \rr$.

This type of arguments apply also in the rest of the paper, whenever one needs to commute $D^s:=(-\Delta)^s$ with the integral $\int_0^t$.

\medskip

\textit{Step II.} In order to extend the range of $s$ we first recall the  chain rule for fractional derivatives (see \cite[Prop. 5 (a)]{MR1878630}, \cite[Prop. 3.1]{MR1124294}). For any $0<s<1$ and $F\in C^1(\R)$ the following inequality holds
\begin{equation}
\label{frac.chain}
  \||D|^s F(u)\|_{L^p(\R)}\lesssim \|F'(u)\|_{L^{p_1}(\R)} \| |D|^s u\|_{L^{p_2}(\R)},
\end{equation}
where
  $1< p, p_2<\infty$, $1<p_1 \leq \infty$ and $\frac 1p=\frac 1{p_1}+\frac 1{p_2}$.

Let us now choose two positive numbers $s_1$ and $s_2$ such that $s_1<\alpha-1$, $s_2<1$ and denote $s=s_1+s_2$. Applying estimate \eqref{frac.chain} to $F(u)=|u|^{q-1}u\in C^1(\rr)$  with $p_1=\infty$, $p_2=p$,   we obtain
\begin{align*}
 \| |D|^su(t) \|_{L^p(\R)}&\leq \| |D|^sK^\alpha_t\|_{L^1(\R)}\| u_0 \|_{L^p(\R)}+\int _0^t
 \||D|^{s_1}\partial _x K^\alpha_{t-\sigma}\|_{L^1(\R)} \|  |D|^{s_2}f(u(\sigma))\|_{L^p(\R)}d\sigma\\
 &\lesssim  t^{-\frac s\alpha}+ \int _0^t (t-\sigma)^{-\frac{s_1+1}\alpha}  \| |D|^{s_2} u(\sigma)\|_{L^p(\R)}.
\end{align*}
Assuming that $\| |D|^{s_2} u(t)\|_{L^p(\R)}\lesssim  t^{-\frac {s_2}\alpha}$ for all $t\in (0,T)$ we obtain  that for any $s<\alpha-1+s_2$ we have
\[
   \| |D|^su(t) \|_{L^p(\R)}\lesssim  t^{-\frac s\alpha} + \int _0^t (t-\sigma)^{-\frac{s_1+1}\alpha}   \sigma ^{-\frac {s_2}\alpha}d\sigma\lesssim t^{-\frac s\alpha}, \quad\ \forall t\in (0,T).
\]
This means that we always we can gain up to $\alpha-1$ derivatives with respect to the initial assumption.

Repeating the above argument and using Step I we obtain that for any $s\in (0,\alpha)$ and any $p\in (1,\infty)$ we have $u(t)\in H^{s,p}(\R)$ for all $t\in (0,T)$ and
\[
  \|| D|^s u(t)\|_{L^p(\R)}\lesssim t^{-\frac s{\alpha}}, \quad \forall \ t\in (0,T).
\]
Moreover, using the properties of the Hilbert transform we also obtain for any $s\in [0,\alpha-1)$ and any $p\in (1,\infty)$
\[
  \|| D|^{s} u_x(t)\|_{L^p(\R)}\lesssim t^{-\frac s{\alpha}}, \quad \forall \ t\in (0,T).
\]

\textit{Step III.} Let us now consider the case $s\geq \alpha$. We write the equation for $u_x$:
\[
  u_x(t)=\partial_x K^\alpha_t\ast u_0+\int _0^t \partial_x(K^\alpha_{t-\sigma})\ast f'(u)u_x(\sigma)d\sigma.
\]
Let us consider $s=s_1+s_2$ with $0<s_1<\alpha-1$ and $0<s_2<\min\{\alpha,q\}-1$. Thus
\begin{align*}
  \| |D|^{s_1+s_2}u_x(t) \|_{L^p(\R)}&\leq \| |D|^{s_1+s_2}\partial_x K_t\|_{L^1(\R)}\|u_0\|_{L^p(\rr)}\\
  &\quad +
\int _0^t \| |D|^{s_1} \partial_x K_{t-\sigma}\|_{L^1(\rr)} \| |D|^{s_2} (f'(u)u_x)\|_{L^p(\rr)}d\sigma\\
&\lesssim t^{-\frac {s+1}\alpha} +\int _0^t (t-\sigma)^{-\frac{1+s_1}\alpha}  \| |D|^{s_2} (f'(u)u_x)\|_{L^p(\rr)}d\sigma.
\end{align*}
Leibniz's rule (\cite[Th. 3]{MR1878630}, \cite[Prop. 3.3]{MR1124294}) gives us that
\[
  \| |D|^{s_2} (f'(u)u_x)\|_{p}\lesssim \| |D|^{s_2} f'(u)\|_{p_1}\|u_x\|_{p_2}+ \||D|^{s_2}u_x\|_{q_1} \|f'(u)\|_{q_2}
\]
where $\frac 1p=\frac 1{p_1}+\frac 1{p_2}=\frac 1{q_1}+\frac 1{q_2}$ and $1<p_1,q_1<\infty$, $1<p_2,q_2\leq \infty$ (Th. 3 in \cite{MR1878630} allows the case $p_2=q_2=\infty$). Choosing $q_1=p$, $q_2=\infty$ we obtain
\begin{align*}
    \| |D|^{s_1+s_2}u_x(t) \|_{L^p(\R)}&\lesssim t^{-\frac {s+1}\alpha} +I_1+I_2,
    \end{align*}
where
\[
  I_1=\int _0^t (t-\sigma)^{-\frac{1+s_1}\alpha}\| |D|^{s_2} f'(u(\sigma))\|_{p_1}\|u_x(\sigma)\|_{p_2}d\sigma
\]
and
\[
  I_2=\int _0^t (t-\sigma)^{-\frac{1+s_1}\alpha}\||D|^{s_2}u_x(\sigma)\|_{p} \|f'(u(\sigma))\|_{\infty}d\sigma.
\]
For $s_2<\alpha-1$, using Step II,   we have for any $t\in (0,T)$
\[
I_2\lesssim
   \int _0^t (t-\sigma)^{-\frac{1+s_1}\alpha} \sigma^{-\frac{1+s_2}\alpha}d\sigma\simeq t^{1-\frac 1\alpha -\frac{s+1}\alpha}\lesssim t^{  -\frac{s+1}\alpha}.
\]
It remains to estimate the first term. For $u_x$ we use the estimates from the previous step since $1<\alpha$ to obtain that $\|u_x(\sigma)\|_{p_2}\lesssim \sigma^{-\frac 1{\alpha}}$. For the term $|D|^{s_2} f'(u)$ we use the fact that $f'(u)=q|u|^{q-1}$ is H\"older continuous of order $q-1$ so for $s_2,\beta$ satisfying
 \[0<s_2<q-1<1, \quad  0<\frac{s_2}{q-1}<\beta <1,\]
  we have \cite[Proposition A.1]{MR2318286}
 \[
  \| |D|^{s_2}|u|^{q-1}\|_{p_1}\leq \| |D|^\beta u\|_{r_2}^{s_2/\beta} \| |u|^{q-1-\frac {s_2}\beta} \|_{r_3}
\]
where
\[
  \frac 1{p_1}=\frac{s_2}{ r_2\beta}+\frac 1{r_3},\quad  r_3\Big(1-\frac{s_2}{(q-1)\beta}\Big)>1.
\]
Choosing  $r_3$ large enough such that
\[
  r_3\Big( (q-1)-\frac{s_2}\beta \Big)\geq 1
\]
the last condition is satisfied and moreover the term $\| |u|^{q-1-\frac {s_2}\beta} \|_{r_3}$ belongs to $L^\infty((0,T))$
 since $u  \in L^\infty((0,T),L^1(\rr)\cap L^\infty(\rr))$. On the other hand
for $\beta <1$ we have estimates on the term $|D|^\beta u$ in the $L^{r_2}(\rr)$-norm, $r_2>1$, obtained previously. This gives us that
\begin{align*}
  I_1& \lesssim
     \int _0^t (t-\sigma)^{-\frac{1+s_1}\alpha} \| |D|^\beta u(\sigma)\|_{r_2}^{s_2/\beta} \|u_x(\sigma)\|_{p_2}d\sigma\\
   &  \lesssim \int _0^t (t-\sigma)^{-\frac{1+s_1}\alpha} \sigma ^{-\frac \beta \alpha \frac {s_2}\beta} \sigma ^{-\frac{1}\alpha}d\sigma
 = \int _0^t (t-\sigma)^{-\frac{1+s_1}\alpha} \sigma ^{-\frac{s_2+1}\alpha}d\sigma<\infty
\end{align*}
since $s_1<\alpha-1$ and $s_2<\min\{\alpha,q\}-1\leq q-1$.  To do that we have to check that for any fixed $p\in (1,\infty)$, $s_2\in (0,q-1)$ and $q\in (1,2)$ the following system has a solution $(p_1,\beta,r_2,r_3)$
\[
 p\leq p_1<\infty, \, \frac 1{p_1} =\frac{s_2}{\beta r_2}+\frac {1}{r_3}, \,  \frac{s_2}{q-1}<\beta <1, \,   (q-1)-\frac{s_2}\beta  \geq \frac 1{r_3}, \, r_2>1.
\]
In order to show the existence of $\beta, r_2,r_3,p_1$ which solves the above system we proceed as follows: Given $s_2\in (0,q-1)$ let us choose $\beta$ such that
\[
  \frac{s_2}{q-1}<\beta<1.
\]
We now choose $r_2\geq 2p$ and $r_3$ such that
\[
  r_3\geq \max\left\{2p,\frac{1}{q-1-\frac{s_2}\beta}\right\}.
\]
Thus we choose $p_1$ such that
\[
   \frac 1{p_1} =\frac{s_2}{\beta r_2}+\frac {1}{r_3}< \frac{q-1}{r_2}+\frac 1{r_3}\leq
  \frac 1{r_2}+\frac 1{r_3}\leq \frac 1p.
\]
The choice of $r_2$ and $r_3$ guarantees that $p_1\geq p$.

As a consequence of the above estimates for any  $s_2<\min\{q,\alpha\}-1 $ we can always make such a choice. Then we obtain that $u\in H^{s,p}$ for any $s<1+\alpha-1+\min\{q,\alpha\}-1=\alpha+\min\{q,\alpha\}-1$ and $1<p<\infty$.
\end{proof}

\begin{proposition} \label{PropExistEps}
Assuming that the initial data is positive and bounded  $u_0 \geq \epsilon>0$ then the unique mild solution of Problem \eqref{Problem1} satisfies \\
(i) $u(t,x)$ is also positive and bounded with $\epsilon \le u(t) \le \|u_0\|_{L^\infty(\rr)}$, for all $x \in \R$.\\
(ii) $u\in C_{b}^\infty ((0,\infty)\times \R)$.
\end{proposition}

\begin{proof}
	Using the maximum principle in Proposition \ref{PropExistence} we have that $\epsilon\leq u(t)\leq  \|u_0\|_{L^\infty(\rr)}$ for all $t>0$. This gives us that the nonlinearity $f(s)=s^q/q$ belongs to $C^\infty((\epsilon, \|u_0\|_{L^\infty(\rr)}))$ and then the results of \cite[Proposition 5.1, Theorem 5.2]{Droniou} guarantee that
	$u\in C_{b}^\infty((0,\infty)\times \rr)$.
		\end{proof}

\subsection{Smooth approximate solutions}
Some of the estimates we need to prove in this paper require positive solutions. This is why we proceed by considering approximating the problem with positive data which, thanks to the maximum principle, also admits positive solutions. We will prove the necessary estimates for the approximating problem and then pass to the limit. Let $u_0\in L^\infty(\rr)$ nonnegative be the initial data of Problem \eqref{Problem1}. We consider the following approximating problem
\[
  \left\{ \begin{array}{ll}
  (u_\epsilon)_{t}(t,x) + (-\Delta)^{\alpha/2} u_\epsilon(t,x)+|u_\epsilon|^{q-1}(u_\epsilon)_x=0 &\text{for }t>0 \text{ and } x \in \mathbb{R}, \\[10pt]
  u_\epsilon(0,x)  =u_{0,\epsilon}(x) &\text{for } x \in \mathbb{R},
    \end{array}
    \right.
    \tag{$P_\epsilon$} \label{Problem1eps}
    \]
where $u_{0,\epsilon}$ is an approximation of $u_0$.

\begin{lemma}\label{lemma_ueps_u}
  Let $u$ be the solution of Problem \eqref{Problem1} with initial data  $u_0 \ge 0$ and let   $ u_\epsilon $ be the solution
  of Problem \eqref{Problem1eps} with initial data $u_{0,\epsilon}=u_0 + \epsilon$.
    Then for every $T>0$ we have
 \begin{equation}\label{Conv_ueps_u}
     \max _{t\in [0,T]}\|u_\epsilon (t) - u(t)\|_{L^\infty(\R)} \to 0 \quad \text{as} \quad \epsilon \to 0.
 \end{equation}
 \end{lemma}
\begin{proof}
Proposition \ref{PropExistEps} shows that there  exists a unique mild solution of Problem   \eqref{Problem1eps}     with
$u_\epsilon\in C_{b}^\infty ((0,\infty)\times \R)$
 and
  $\epsilon \le  u_\epsilon (t,x) \le \|u_0\|_{L^\infty(\rr)}+\epsilon$ for all $x\in \R$, $t\ge 0.$

For $u_0\geq 0$ the maximum principle in Proposition \ref{PropExistence} guarantees that $u$ the solution of system \eqref{Problem1} is also nonnegative. Let us choose $\epsilon\leq \|u_0\|_{L^\infty(\rr)}$ and $A=2\|u_0\|_{L^\infty(\rr)}$.
The result follows from the fact that $f(s)=s^{q}/q$ is Lipschitz on $[0,A]$ and the use of Fractional Gronwall Lemma \cite[Lemma 2.4]{Bouharguane2013}.
Indeed, using the mild formulation we find that
\[
  u(t)-u_\epsilon(t)=K_{t}^\alpha\ast (u_0-u_{0,\epsilon})+\int _0^t (K^{\alpha}_{t-s})_x \ast (f(u(s))-f(u_\epsilon(s)))ds.
\]
Then
\begin{align*}
    \| u(t)-u_\epsilon(t)\|_{L^\infty(\rr)}&\leq  \|K_{t}^\alpha\|_{L^1(\rr)} \|u_0-u_{0,\epsilon}\|_{L^\infty(\rr)}\\
    &\quad+\int _0^t \|K_{t-s}^\alpha\|_{L^1(\rr)} \| f(u(s))-f(u_\epsilon(s)) \|_{L^\infty(\rr)}ds\\
    &\leq \epsilon + C A^{q-1} \int _0^t (t-s)^{-\frac 1\alpha} \| u(s)-u_\epsilon(s)\|_{L^\infty(\rr)}ds.
\end{align*}
Since $\alpha>1$ we can apply Fractional Gronwall Lemma \cite[Lemma 2.4]{Bouharguane2013} to obtain that for any $T>0$
 there exists a positive constant $C(T)$ such that
 \[
    \| u(t)-u_\epsilon(t)\|_{L^\infty(\rr)}\leq   \epsilon\, C(T), \quad \forall \ t\in [0,T].
\]
This finishes the proof.
\end{proof}

\subsection{Hyperbolic estimates for \eqref{Problem1eps}}

For any $\epsilon>0$ we now consider initial data in Problem \eqref{Problem1eps} a function  $u_{0,\epsilon}$  such that $\epsilon\leq u_{0,\epsilon}\leq m$ and let $u_\epsilon$ be the solution of Problem \eqref{Problem1eps}.
The following is the key estimate towards the proof of the asymptotic result.

\begin{proposition}\label{prop.oleinik}
Let $1<q,\alpha\leq  2$. For any $\epsilon>0$ solution $u_\epsilon$   of Problem \eqref{Problem1eps}  satisfies the \emph{Oleinik type estimate}:
\begin{equation}\label{OleinikEps}
(u_\epsilon ^{q-1})_ x (t,x)\le \frac{1}{t}, \quad \forall t>0, \, x\in \R.
\end{equation}
\end{proposition}

\begin{remark}We emphasize here that the result holds for all $q,\alpha\in (1,2]$ without the assumption $q<\alpha$. When $\alpha=2$ this estimate has been obtained in \cite{EVZArma}.
A similar result has been proved in \cite{AlibaudAndreianov}  when $\alpha\in (0,1)$ and $q=2$ for the regularised equation
\[
  u_{t} + (-\Delta)^{\alpha/2} u+|u|^{q-1}u_x-\epsilon u_{xx}=0.
\]
We are not able to use the barrier method  as in \cite{AlibaudAndreianov}. The difficulty comes from the fact that one should prove that for a suitable function, i.e. $\Phi(x)=(1+x^{2})^{\gamma}$, the term
 \[
A(w,z)=- (2-q)w \La [z^{\beta+1}] + z\La [ z^{\beta}w]
\]
satisfies $z^{-(\beta+1)}(t,x)A(\Phi(x),z(t,x))\geq -C_{z} $ for all $x\in \rr$ and $t>0$ where $z$ is a $C_b^{\infty}((0,\infty)\times \rr)$ function and $\beta=\frac{2-q}{q-1}>0$.
 Observe that in the case $q=2$ we have $\beta=0$,  $A(w,z)=\La w$ and the required estimate holds by choosing $\gamma$ suitably.
\end{remark}
 \begin{proof}We consider $\alpha\in (1,2)$ since the case $\alpha=2$ has been treated in \cite{EVZArma}.
 Let $z(t,x)=(u_\epsilon)^{q-1}(t,x)$. For simplicity we will not make explicit the dependence on $\epsilon$.
  Then $z\in C_b^{\infty}((0,\infty)\times \rr)$ and
  $$
  z_t + (q-1)z^{1- \frac{1}{q-1}}  \La [z^{\frac{1}{q-1}}] + z z_x=0.
 $$
 Let $w(t,x)=z_x(t,x)$. Then $w\in  C_b^{\infty}((0,\infty)\times \rr)$ and it verifies
 \begin{equation*}
 w_t + w^2 + zw_x +z^{-\beta -1} A(w,z)=0.
\end{equation*}

 We  continue as in \cite{CazacuIgnatPazoto} following some ideas from \cite{Droniou,KarchQualitProp2009}.
Let us denote $W(t)=\sup _{x\in \rr} w(t,x)$. Since $z$ is $C^k_b((0,\infty)\times \rr)$ using the same  arguments as in \cite[Th. 1.18]{KarchQualitProp2009} we have  that $W$ is locally Lipschitz. In particular $W$ is absolutely continuous so differentiable almost everywhere.
We now differentiate $W(t)$ for $t>0$ and obtain the equation it satisfies. Let us choose $0< s<t$. We use Taylor's expansion in the time variable $t$:
\[
   w( t,x)\leq w(t-s,x)+ s w_t( t,x) + C s^2\leq W(t-s)+s w_t ( t,x) + C s^2.
\]
It follows that
\begin{equation}
\label{eq.w.1}
  w(t,x)+s \Big(w^2(t,x)+ z w_x(t,x) +z^{-\beta-1}(t,x)   A(w(t,x),z(t,x))\Big)\leq W(t-s)+Cs^2.
\end{equation}

Let us fix $t>0$ and consider the points $x_n$ such that $w(x_n , t)=W(t)-1/n$. Following  \cite[Lemma~1.17]{KarchQualitProp2009}  we have
\begin{equation*}
  \lim _{n\rightarrow\infty} w_x(t,x_n)= 0.
\end{equation*}
Moreover, since the sequence $(z(t,x_n))_{n\geq 1}$ is bounded we can assume that, up to a subsequence, $z(t,x_n)\rightarrow p(t)$ for some  function $p(t)\in [\epsilon,m]$.

Now we evaluate \eqref{eq.w.1} at the point $x=x_n$. Letting $n\rightarrow\infty$ we can easily see that, up to a subsequence,
\[
  w(t,x_n)+s( w^2(t,x_n)+zw_x(t,x_n))\rightarrow  W(t)+ s W^2(t).
\]
We claim that up to a subsequence
\begin{equation}
\label{below.A}
  A(w(t,x_n),z(t,x_n))\geq W(t) I_n(t)-o(1)
\end{equation}
 for some bounded non-negative  sequence $I_n(t)$. This implies that, up to a subsequence, $I_n(t)\rightarrow q(t)$ where $q(t)\geq 0$.
 This implies that inequality \eqref{eq.w.1} becomes
 \[
  W(t)+s\big(W^2(t) +p^{-\beta-1}(t) q(t) W(t)\big)\leq W(t-s)+Cs^2.
\]
Letting $s\rightarrow 0$ we obtain that for a.e. $t>0$,  $W$ satisfies
\[
  W'(t)+W^2(t)+p^{-\beta-1}(t) q(t) W(t)\leq 0.
\]
Now it follows using classical ODEs arguments  (see for example \cite[p. 3136]{CazacuIgnatPazoto})  that $W$ satisfies
\[
  \max\{W(t),0\}\leq \frac 1t, \quad \forall \ t>0.
\]

To finish the proof it remains to prove claim \eqref{below.A}.
To do that, we use representation \eqref{frac.lap} with suitable $r=r_n$ depending on $x_n$ that will be specified latter. Using that $\beta/(\beta+1)=2-q$  we write $A(w,z)$ as follow
\begin{align*}
  A&(w(x),z(x))/c(\alpha)\\
  &=-z(x)\int _{|y|>r}\frac{z^\beta w(x+y)-z^\beta w(x)}{|y|^{\alpha+1}}dy
  -z(x)\int _{|y|<r}\frac{z^\beta w(x+y)-z^\beta w(x) - y (z^\beta w)_x(x)}{|y|^{\alpha+1}}dy\\
  &+\frac{\beta}{\beta+1} w(x)\int _{|y|>r}\frac{z^{\beta+1} (x+y)-z^{\beta+1}(x)}{|y|^{\alpha+1}}dy \\
  & +
  \frac{\beta}{\beta+1} w(x)\int _{|y|<r}\frac{z^{\beta+1}(x+y)-z^{\beta+1}(x) - y (z^{\beta+1})_x(x)}{|y|^{\alpha+1}}dy\\
  &=\int _{|y|>r}\Big[ w(x)\big( \frac{z^{\beta+1}(x)}{\beta +1}+\frac{\beta z^{\beta+1}(x+y)}{\beta +1}\big) -w(x+y)z^{\beta}(x+y)z(x)\Big]\frac{dy}{|y|^{\alpha+1}}+R(r,w,z),
\end{align*}
where we have collected the integrals in the ball of radius $r$ in the reminder term $R$. It is easy to evaluate each integral in $R$ and prove that
\[
  |R(t,w,z)|\lesssim r^{2-\alpha} (\|z\|_{\infty} \| (z^\beta w)_{xx})\|_{\infty}+\|w\|_{\infty} \| (z^{\beta+1})_{xx}\|_\infty )\leq C(\epsilon,\beta, \|z(t)\|_{C^3_b(\rr)})
  r^{2-\alpha}.
\]
Let us evaluate $A(w,z)$ at the point $x=x_n$. Using that $w(t,x_n)=W(t)-1/n$ we obtain
\begin{align*}
  A&(w(t,x_n),z(t,x_n) )\\
  &\geq \int _{|y|>r}\Big[ w(t,x_n)\big( \frac{z^{\beta+1}(t,x_n)}{\beta +1}+\frac{\beta z^{\beta+1}(t,x_n+y)}{\beta +1}\big) \\
  &\qquad\qquad\qquad\qquad\qquad\qquad -w(t,x_n+y)z^{\beta}(t,x_n+y)z(t,x_n)\Big]\frac{dy}{|y|^{\alpha+1}} -Cr^{2-\alpha}\\
  &=\int _{|y|>r}\Big[ W(t)\big( \frac{z^{\beta+1}(t,x_n)}{\beta +1}+\frac{\beta z^{\beta+1}(t,x_n+y)}{\beta +1}\big) -w(t,x_n+y)z^{\beta}(t,x_n+y)z(t,x_n)\Big]\frac{dy}{|y|^{\alpha+1}}\\
  &\quad -\frac 1n \int _{|y|>r}   \big( \frac{z^{\beta+1}(x_n)}{\beta +1} +\frac{\beta z^{\beta+1}(x_n+y)}{\beta +1}\big)\frac{dy}{|y|^{\alpha+1}}-Cr^{2-\alpha}\\
  &\geq W(t) \int _{|y|>r}\Big[ \big( \frac{z^{\beta+1}(t,x_n)}{\beta +1}+\frac{\beta z^{\beta+1}(t,x_n+y)}{\beta +1}\big) -z^{\beta}(t,x_n+y)z(t,x_n)\Big]\frac{dy}{|y|^{\alpha+1}}\\
  &\quad-Cr^{2-\alpha}-\frac{\|z(t)\|_\infty^{\beta+1}}{nr^\alpha}\\
  &:=W(t)I_n(z(t),r,x_n) -Cr^{2-\alpha}-\frac{C}{nr^\alpha}.
\end{align*}
Let us now choose $r=r_n$ such that $r_n\rightarrow 0$ and $nr_n^\alpha\rightarrow\infty$ as $n\rightarrow\infty$. Lemma \ref{positive.term} below shows that $I_n(t)=I(z(t),r_n,x_n)$ is well defined and is uniformly bounded. Moreover, $I_n(t)\geq 0$: indeed this follows by applying Young's inequality $ \frac{a^p}{p}+\frac{b^q}{q}\ge ab$, $\frac{1}{p}+\frac{1}{q}=1$ for $a=z(t,x_n)$, $b=z(t,x_n+y)$, $p=\beta+1$, $q=\frac{\beta+1}{\beta}$.
Hence
\[
  A(w(t,x_n),z(t,x_n)\geq W(t)I_n(t)-o(1)
\]
and claim \eqref{below.A} is proved. The proof is now complete.
 \end{proof}

 \begin{lemma}
\label{positive.term} Let $z\in C^1_b(\rr)$ such that $0<\epsilon\leq z\leq m$ and $\alpha\in (0,2]$, $\beta>0$.
The  function
\[
I(z,r,x)= \int _{|y|>r}\left(\frac 1{\beta+1} z^{\beta+1}(x)+\frac{\beta}{1+\beta}z^{\beta+1}(x+y) -z(x)z^{\beta}(x+y)\right)\frac{dy}{|y|^{1+\alpha}},
\]
defined for $r>0, \ x\in \rr,$	satisfies
	\[
  |I(z,r,x)|\leq C(\beta,\epsilon,m) \|z\|_{C^1_b(\rr)}^2.
\]
\end{lemma}
\begin{proof}
Observe that for any $\beta>0$ we have $\beta t^{\beta+1} +1-(\beta+1)t^{\beta}\sim (t-1)^2$ as $t\sim 1$. Then the following inequality holds
\[
  |\beta t^{\beta+1} +1-(\beta+1)t^{\beta}|\leq C(\beta) \max\{1,t^{\beta-1}\} |t-1|^2, \quad \forall t>0.
\]
Applying to $t=z(x+y)/z(x)$ and integrating on $y$ we obtain that
\begin{align*}
  \int _{|y|>r}& \left(\frac 1{\beta+1} z^{\beta+1}(x)+\frac{\beta}{1+\beta}z^{\beta+1}(x+y) -z(x)z^{\beta}(x+y)\right)\frac{dy}{|y|^{1+\alpha}} \\
  &\leq C(\beta, \epsilon,m) \int _{|y|>r} \frac{|z(x+y)-z(x)|^2}{|y|^{1+\alpha}}dy\\
  &\leq C(\beta, \epsilon,m) \Big(\|z_x\|_{L^\infty(\rr)}^2\int _{|y|<1}\frac 1{|y|^{\alpha-1}}dy+\|z\|_{L^\infty(\rr)}^2 \int _{|y|>1}\frac 1{|y|^{\alpha+1}}dy\Big).
\end{align*}
The proof is now complete.
\end{proof}

\subsection{Estimates for the solution of Problem \eqref{Problem1} }

 We will prove  various estimates for $u$ the mild solution of Problem  \eqref{Problem1} by using as the starting point the estimate in Proposition \ref{prop.oleinik}.
  We recall that $u \in C((0,\infty), H^{s,p}(\R))$ for any $s<\alpha+q-1$ and $1<p<\infty$,
according to Proposition \ref{PropExistence}.
Remark that  \eqref{Conv_ueps_u} and the regularity of $u$ implies that $u_\epsilon (t,x) \to u(t,x)$ for all $t>0$, $x\in \R$, where 
$u_\epsilon$ is the solution of  Problem \eqref{Problem1} with initial data $u_{0,\epsilon}=u_0+\epsilon$.

\begin{lemma}\label{est.for.u}
 Let $u$ be the solution of Problem \eqref{Problem1} with nonnegative initial data $u_0\in L^1(\R)\cap L^\infty(\rr)$. Then the following estimates hold:
 \begin{enumerate}
 \item  Mass conservation: $\int_{\R}u(t,x)dx= M, \quad \forall t\ge 0.$
  \item Hyperbolic estimate: $\displaystyle
(u^{q-1})_ x (t,x)\le \frac{1}{t}$ for all $ t>0 $ in $\mathcal{D'}(\rr)$.
   \item\label{Linf}Upper bound: $\displaystyle 0\leq u(t,x) \le \left(\frac{q}{q-1}M \right)^{1/q}  t^{-1/q} $ for all $t>0,\ x\in \R.$
   \item Decay of the $L^p$-norm, $1\leq p\leq \infty$:
   \[\displaystyle \|u(t,\cdot)\|_{L^p(\R)}\le \left(\frac{q}{q-1}\right)^\frac{p-1}{pq}M^{\frac{p-1}{pq}+\frac{1}{p}}\, t^{-\frac{1}{q}\left(1-\frac{1}{p}\right)},\  \forall \ t>0.
   \]
   \item  Decay of the spatial derivative: $ \displaystyle u_x(t,x) \le C(q) M ^{\frac{2-q}{q}}  t^{-\frac{2}{q}}$  for all $ t>0 $, a.e. $x\in \rr$.
   \item  $W^{1,1}_{\text{loc}}(\R)$ estimate:
   \[
   \int_{|x|\leq R} |u_x(t,x)|dx \le  2R\,  C(q) M ^{\frac{2-q}{q}}  t^{-\frac{2}{q}} + 2\left(\frac{q}{q-1}M \right)^{1/q}  t^{-1/q},\quad \forall  t>0.
   \]
   \item Energy estimate: for every $0<\tau<T$,
  \begin{equation*}
  \displaystyle \int_\tau^T \int_\R |(-\Delta)^{\alpha/4} u(t,x)|^2 dx dt \le \frac{1}{2}\int_{\R}u^2(\tau,x)dx \le \frac{1}{2}\left(\frac{q}{q-1} \right)^{1/q}  \tau^{-1/q} M^{\frac{q+1}{q}}.
  \end{equation*}
 \end{enumerate}
\end{lemma}

\begin{proof} 
Using  the regularity obtained in
Proposition \ref{PropExistence} ii), we can integrate the integral representation \eqref{mild.form} we respect to the $x$ variable.  Using Fubini's theorem, we obtain the mass conservation property.
Alternatively, the mass conservation also follows from the distributional formulation. In fact,   a classical approximation argument allows to write for any $\psi\in C_c^2(\rr)$  the following identity
\[
  \int _{\rr} u (t,x)\psi (x)dx-\int _{\rr}u (0,x)\psi(x)dx=\int _0^t \int _{\rr} f(u )\psi_x-\lambda^{q-\alpha}\int _0^t \int _{\rr} u  (-\Delta)^{\alpha/2}\psi.
\]
We choose as test function $\psi_R(x)=\psi(x/R)$ where $\psi\in C_c^2(\rr)$, $0\leq \psi\leq 1$,
$\psi(x)\equiv 1$ for $|x|\leq 1 $ and $\psi(x)\equiv 0$ for $|x|\geq 2$. Then $(\psi_R)_x=O(R^{-1})$ and $ (-\Delta)^{\alpha/2}\psi_R =O (R^{-\alpha}).$   Letting $R\rightarrow \infty$ gives us the conservation of the mass.

For the second property we consider $u_\epsilon $   the solution of Problem \eqref{Problem1eps} with
$u_{0,\epsilon}=u_0+\epsilon$. Then by Lemma \ref{lemma_ueps_u} we have that $u_\epsilon(t) \to u(t)$ in $ L^\infty(\R)$.
This way we are able to pass to the limit estimate \eqref{OleinikEps} in a distributional sense.

The regularity results obtained in Proposition
\ref{PropExistence} show that $u(t)$ is a continuous function for any $t>0$.
 Using estimate \eqref{OleinikEps} for $u_\epsilon$ and letting $\epsilon\rightarrow 0$ imply that
\begin{equation}\label{ineq1}
    u^{q-1}(t,x)-u^{q-1}(t,y)\leq \frac{x-y}t, \quad \forall \, y<x, \quad \forall \, t>0.
  \end{equation}
 The proof of the third estimate follows from \eqref{ineq1}: we fix $x\in \R$ and we integrate  in $y$ on the interval $I=\{y \in \R: y<x \text{ and } u^{q-1}(t,x)-\frac{x-y}t \ge 0\}$. Thus
\begin{align*}
M&=\int_{\R}u(t,y)dy \ge \int_{I} \left(u^{q-1}(t,x)-\frac{x-y}t \right)^{1/(q-1)}dy =\frac{1}{t^{1/(q-1)}}\int_0^{t u^{q-1}(t,x)} z^{1/(q-1)} dz \\
&=\frac{q-1}{q} t u^{q}(t,x).
\end{align*} 
Inequality (iv) is a consequence of the mass conservation and previous estimate.

Using the intermediate value theorem we obtain that
 \[
 u(t,x)-u(t,y)=	\left( u^{q-1}(t,x)-u^{q-1}(t,y) \right) \frac{1}{q-1}\xi^{2-q} ,
\]
for some $\xi$ between $u(x)$ and $u(y)$.  Then according to \eqref{ineq1} for any $y<x$ the following holds
\[
 u(t,x)-u(t,y)\le 	\frac{1}{q-1}\|u(t)\|_{L^\infty}^{2-q} \, \frac{x-y}{t}.
\]
Then using the upper bound from point \eqref{Linf} we get
\[
\frac{ u(t,x)-u(t,y)}{x-y}\le 	 C(q) M^\frac{2-q}{q}t^{-\frac{2}{q}}.
\]
Since $u$ is differentiable a.e. we can let $y\to 0$ we obtain the desired upper bounds for $u_x$.

Denoting $B_R=(-R,R)$ and using that $u\in W^{1,1}_{loc}(\rr)$ we have
\begin{align*}
\int_{B_R}|u_x(t,x)| dx  &= \int_{B_R \cap  \{u_x>0\}} u_x dx  + \int_{B_R \cap \{u_x<0\}}
(-u_x) dx \\
&=2  \int_{B_R \cap \{u_x>0\}} u_x dx  + u(-R) -u(R) \\
&\le 2R\,  C(q) M ^{\frac{2-q}{q}}  t^{-\frac{2}{q}} + 2\left(\frac{q}{q-1}M \right)^{1/q}  t^{-1/q}.
\end{align*}
Multiplying equation \eqref{Problem1} by $u$ and integrating by parts
$$\frac{1}{2}\frac{d}{dt}\int_{\R}{u^2} dx  + \int_{\R} |(-\Delta)^{\alpha/4}u|^2dx =0.$$
The decay of the $L^2(\rr)$-norm gives that
\begin{align*}
\int_{\tau}^T\int_{\R} |(-\Delta)^{\alpha/4}u|^2dx dt &\le   \frac{1}{2} \int_{\R}{u^2(\tau)} dx \leq \frac{1}{2}\left(\frac{q}{q-1} \right)^{1/q}  \tau^{-1/q} M^{\frac{q+1}{q}}.
\end{align*}
The proof is now finished.
\end{proof}

\section{Asymptotic behaviour}\label{SectAsymp}

Let $u$ be the unique mild  solution to Problem \eqref{Problem1} with nonnegative data $u_0\in L^1(\R) \cap L^\infty(\R)$ obtained in Proposition \ref{PropExistence}. In order to prove the asymptotic behaviour we perform the method  developed by Kamin and V\'{a}zquez in \cite{KamVaz88}.  For every $\lambda>0$, we define the rescaled function
\begin{equation}\label{u_lambda}
u_\lambda(t,x) := \lambda u(\lambda^{q} t, \lambda x).
\end{equation}
It follows that $\ul$ is a solution of the problem
\[
  \left\{ \begin{array}{ll}
  (\ul)_t + \lambda^{q-\alpha}(-\Delta)^{\alpha/2} [\ul] +(\ul)^{q-1}(\ul)_x=0, & x\in \rr, \, t>0, \\[10pt]
  u_\lambda(0,x)=\lambda u_0(\lambda x), &x\in \rr.
    \end{array}
    \right.
    \tag{$P_\lambda$}\label{Plambda}
\]
Using the properties obtained in Lemma \ref{est.for.u} and the definition of $u_\lambda$ we obtain the following uniform in $\lambda$ estimates for $u_\lambda$.

\begin{lemma}\label{LemmaEstul} Let $\ul$ be the rescaled function defined by \eqref{u_lambda}. Then the corresponding a-priori estimates are true.
\begin{enumerate}
\item Mass conservation: $\int_{\R}u_\lambda(t,x)dx= M, \quad \forall t\ge 0, \, \forall \lambda>0.$
  \item Decay of the $L^p$-norm:  \\
 $\displaystyle \|u_{\lambda}(t,\cdot)\|_{L^p(\R)}\le \left(\frac{q}{q-1}\right)^\frac{p-1}{pq}M^{\frac{p-1}{pq}+\frac{1}{p}}\, t^{-\frac{1}{q}\left(1-\frac{1}{p}\right)}$, $\forall \lambda>0,$\, $\forall p \ge 1$.
\item\label{W11loc} $W^{1,1}_{loc}(\R)$ estimate: for $R>0$ we have\\
$ \displaystyle \int_{B_R} | (\partial_x\ul)(t,x)|dx \le   2R\,  C(q) M ^{\frac{2-q}{q}}  t^{-\frac{2}{q}} + 2\left(\frac{q}{q-1}M \right)^{1/q}  t^{-\frac{1}{q}},\quad \forall\ t>0.$
   \item Energy estimate: for every $0<\tau<T$ and $\lambda>0$
 \begin{equation*}
  \displaystyle  \lambda^{q-\alpha}\int_\tau^T \int_\R |(-\Delta)^{\alpha/4} \ul(t,x)|^2 dx dt \le \frac{1}{2}\int_{\R}\ul^2(\tau,x)dx \le \frac{1}{2}\left(\frac{q}{q-1} \right)^{1/q}  \tau^{-1/q} M^{\frac{q+1}{q}}.
  \end{equation*}
\end{enumerate}
\end{lemma}

In what follows we establish the results stated in Theorem \ref{ThmAsympBehav} by re-writing in an equivalent manner the asymptotic behavior   \eqref{limit.asymp}.
For $1\leq p<\infty$ and $t>0$ we will prove that
\begin{equation}
\label{equiv.limit}
  \| u_\lambda(t,x) - U_M(t,x) \|_{L^p(\R)}\to 0 \quad \text{as}\quad \lambda \to \infty,
\end{equation}
 where $U_M(t,x)$ is the solution to the purely convective equation \eqref{limit.problem}.

We emphasize that  it is enough to prove \eqref{equiv.limit} only for some $t=t_0>0$.

\begin{proof}[Proof of Theorem \ref{ThmAsympBehav}]
For the reader's convenience we divide the proof according to the four-step method developed in \cite{KamVaz88}. Moreover for completeness we recall the following classical compactness argument due to Aubin-Lions-Simon.
\begin{theorem}[\cite{Simon}, Th.5] \label{3spaces}
Let us consider three Banach spaces  $X\hookrightarrow B \hookrightarrow Y$ where $X\hookrightarrow B$ is compact.
Assume $1\leq p\leq \infty$ and \\
i) $\mathcal{F}$ is bounded in $L^p((0,T),X)$,\\
ii) $\|\tau _h f-f\|_{L^p((0,T-h),Y)}\rightarrow 0$ as $h\rightarrow 0$ uniformly for $f\in\mathcal{F}$.

Then $\mathcal{F} $ is relatively compact in $L^p((0,T),B)$ (and in $C([0,T],B)$ if $p=\infty$).
\end{theorem}

Let us  consider $0<t_1<t_2<\infty$.
\medskip

\noindent\textbf{Step I. Compactness of family $(u_\lambda)_{\lambda>0}$ in $C([t_1,t_2],L^2_{loc}(\rr))$}.
Let $B_R =(-R,R)$. We apply the Aubin-Lions-Simon compactness argument in Theorem \ref{3spaces} to the triple $W^{1,1}(B_R) \hookrightarrow L^2(B_R)\hookrightarrow H^{-1}(B_R)$.
Estimate
 \ref{W11loc} in Lemma \ref{LemmaEstul} and the mass conservation give us that  $(u_\lambda)_{\lambda>0} $ is
 uniformly bounded in $L^{\infty}((t_1,t_2):W^{1,1}(B_R))$.  Moreover, we can prove that  $(\partial _t u_{\lambda})_{\lambda>1}$ is uniformly bounded in $L^2 ((t_1,t_2):H^{-1}(B_R))$.
Indeed, let us choose $\varphi \in C_c((0,\infty)\times B_R)$. We extend it with zero outside $B_R$.
For such $\varphi$ and $\lambda>1$ we have
\begin{align*} & \left| \int_{t_1}^{t_2}\int_{\R} (\partial_t u_{\lambda}) \varphi   \right|  \le \left|\int_{t_1}^{t_2}\int_{\R}(u_{\lambda}^q)_x  \varphi   \right| +  \lambda^{q-\alpha}  \left| \int_{t_1}^{t_2}\int_{\R} \La[\ul]    \varphi  \right|  \\
&= \left|\int_{t_1}^{t_2}\int_{\R}u_{\lambda}^q  \varphi_x  \right| +  \lambda^{q-\alpha} \left| \int_{t_1}^{t_2} \int_{\R}(-\Delta)^{\alpha/4}[\ul]   (-\Delta)^{\alpha/4}  \varphi  \right|  \\
&\le \|\ul^q \|_{L^2((t_1,t_2),L^2(\R))} \cdot \|\varphi \|_{L^2((t_1,t_2),H^1(\R))} + \\
& \quad+ \lambda^{q-\alpha} \left(\int_{t_1}^{t_2} \int_{\R}|(-\Delta)^{\alpha/4}\ul |^2 \right)^{1/2} \left(\int_{t_1}^{t_2} \int_{\R}|(-\Delta)^{\alpha/4}\varphi  |^2 \right)^{1/2} \\
&= \|\ul^q \|_{L^2((t_1,t_2),L^2(\R))} \cdot \|\varphi \|_{L^2((t_1,t_2),H^1(\R))} + \\
& \quad + \lambda^{\frac{q-\alpha}{2}} \left(\lambda^{q-\alpha} \int_{t_1}^{t_2} \int_{\R}|(-\Delta)^{\alpha/4}\ul|^2 \right)^{1/2} \left(\int_{t_1}^{t_2} \int_{\R}|(-\Delta)^{\alpha/4}\varphi |^2 dx dt\right)^{1/2} \\
&\le \|\ul^q \|_{L^2((t_1,t_2),L^2(\R))} \cdot \|\varphi \|_{L^2((t_1,t_2),H^1(\R))}  + \lambda^{\frac{q-\alpha}{2}} C(M,q,t_1)  \|\varphi \|_{L^2((t_1,t_2),H^{\alpha/2}(\R))}    \\
&\le C(M,q,t_1) \|\varphi \|_{L^2((t_1,t_2),H^1(\R))}.
\end{align*}
This gives us that
$$\|(u_{\lambda})_t  \|_{L^2 ((t_1,t_2),H^{-1}(B_R))}   \le C(M,q,t_1), \quad \forall \lambda \ge 1.$$
Using the  classical compactness arguments in Theorem \ref{3spaces}, we deduce that $(\ul)_{\lambda>1}$ is relatively compact in $C([t_1,t_2],L^2(B_R))$.  Therefore there exists $U \in C([t_1,t_2],L^2(B_R))$  such that $\ul \to U$ in $ C([t_1,t_2],L^2(B_R)).$
By a diagonal argument we get that  $U \in  C([t_1,t_2], L^2_{loc}(\R))$   and
\begin{equation}\label{conv:ul:Uloc}
\ul \to U \quad \text{in} \quad C([t_1,t_2],L^2_{loc}(\R)) \quad \text{as}\quad  \lambda \to \infty.
\end{equation}

\medskip

\noindent \textbf{Step II. Tail control and convergence in $C([t_1,t_2],L^1(\rr))$.} In view of \eqref{conv:ul:Uloc} we obtain that $\ul \to U$ {in} $C([t_1,t_2]:L^1_{loc}(\R)) $. In order to prove the convergence in $C([t_1,t_2],L^1(\rr))$
we will prove a uniform tail control of the functions $(u_\lambda)_{\lambda>1  }$. More exactly, we prove that there exists a constant $C(M)$ such that
\begin{equation}\label{tail}
\int_{|x|>2 R} \ul(t,x) dx \leq \int_{|x|>R} u_0(x) dx + C( M)\left(
 \frac{t\lambda^{q-\alpha}}{R^\alpha}+ \frac{t^{1/q}}{R}\right),\quad \forall t>0.
\end{equation}
In view of this estimate, classical arguments give us that
\begin{equation*}
\ul \to U \quad \text{in} \quad C([t_1,t_2],L^1(\R)) \quad \text{as}\quad  \lambda \to \infty.
\end{equation*}

Let us now prove estimate \eqref{tail}. Let $\varphi\in C^2(\rr)$ be such that $0\le \varphi\le 1$,  $\varphi\equiv 1$ for $|x|\ge 2$, $\varphi_R \equiv 0$ for $|x|\le 1$. Let $\varphi(x)=\varphi(x/R)$.
Multiplying equation \eqref{Plambda} by $\varphi$ and integrating by parts we obtain
\begin{align*}
\int_{\R}\ul (t) \varphi_R dx
&=\int_{\R}\ul(0)  \varphi_R dx -\lambda^{q-\alpha} \int_0^t\int_{\R} \ul (\tau,x) (-\Delta)^{\alpha/2}  \varphi_R dx d\tau  \\
&\quad +\int_0^t\int_{\R} \ul ^q(\tau,x) (\varphi_R)_x dx d\tau\\
&=I+II+III.
\end{align*}
For $\lambda >1$ the first term satisfies
$$
I \le \int_{|x|\ge R}\ul(0,x)  dx =\int_{|x|>\lambda R} u_0(x) dx \le \int_{|x|> R} u_0(x) dx  .
$$
Using that
 $\varphi\in C^2_b(\rr)$ and the homogeneity of $(-\Delta)^{\alpha/2} $ we obtain that
$$ |((-\Delta)^{\alpha/2}  \varphi_R) (x)| = \frac{1}{R^{\alpha}}|( (-\Delta)^{\alpha/2}  \varphi) (x/R) | \le  \frac{C}{R^{\alpha}}.$$
Thus the second term satisfies
\begin{align*}
II&\le \lambda^{q-\alpha}  \| (-\Delta)^{\alpha/2}  \varphi_R (x) \|_{L^\infty(\R)}\int_0^t\int_{\R} \ul (\tau,x) dx d\tau \le \lambda^{q-\alpha}  \frac{C}{R^{\alpha}} t M.
\end{align*}
The third term is bounded as follows:
\begin{align*}
III&\le \|(\varphi_R)' \|_{L^\infty(\R)} \int_0^t \| \ul (\tau)\|_{L^q(\R)}^q d\tau\le C(M)\frac{t^{1/q}}{R}.
\end{align*}
Using the fact that $\varphi_R$ is identically one outside the ball of radius $2R$ we obtain the desired estimate  \eqref{tail}.

\medskip

\noindent \textbf{Step III. Identifying the limit.} We now prove that  $U\in C ((0,\infty), L^1(\rr))$ obtained above is an entropy solution of system \eqref{limit.problem}.
 First, by construction in \cite{AlibaudEntropy,Droniou}, $u$ is an entropy solution  of Problem \eqref{FHE} and this implies that $\ul$ is an entropy solution of Problem \eqref{Plambda}.
In view of Definition \ref{DefnEntropySol} with the
 particular choice $\eta_{k}(s)=|s-k|$ and $\phi_{k}(s)=\sgn(s-k){(f(s)-f(k))}$, function $u_\lambda$ satisfies for any $\varphi\in C_{c}^{\infty}( (0,\infty)\times \R)$ the following inequality:
\begin{align*}
  \int _0^\infty\int_{\R} &(|u_{\lambda}-k|\partial_t \varphi +\sgn(u_{\lambda}-k)(f(u_{\lambda})-f(k))\partial _x\varphi)dx dt \\
\nonumber  &+c(\alpha)  \lambda^{q-\alpha} \int _0^\infty\int_{\R}\int _{|z|> r} \sgn (u_{\lambda }(t,x)-k)
\frac{ u_\lambda(t,x+z)-u_\lambda(t,x)}{|z|^{1+\alpha}} \varphi(t,x) dzdxdt+\\
\nonumber &+c(\alpha)  \lambda^{q-\alpha} \int _0^\infty\int_{\R} \int _{|z|\leq r} |u_{\lambda }(t,x)-k|
\frac{\varphi(x+z)-\varphi(x)- \varphi'(x)z}{|z|^{1+\alpha}}dzdxdt\geq 0.
\end{align*}
We prove that the last two terms, denoted by $I_{1},I_{2}$, tend to zero as $\lambda\rightarrow \infty$. Assume that $\varphi$ is supported in $(0,T)\times (-R,R)$ for some positive $T$ and $R$.
 The first term satisfies
\begin{align*}
|I_1|&\leq 2c(\alpha) \lambda^{q-\alpha} \|\varphi\|_{L^\infty(\R)} \int_0^T \int _{\R} |u_\lambda(t,x)|\int _{|z|>r} \frac1{|z|^{1+\alpha}}dz\\
&\leq C(\alpha,r,\varphi) T M \lambda^{q-\alpha}\rightarrow 0, \quad \lambda\rightarrow\infty.
\end{align*}
 In the case of the second term we have
 \begin{align*}
  |I_{2}|\leq c(\alpha)\lambda^{q-\alpha} \|\varphi''\|_{L^\infty(\R)} \int _0^T \int _{|x|\leq R+r} |u_\lambda(t,x)-k| \int_{|z|\leq r} \frac1{|z|^{\alpha-1}}dz \lesssim \lambda^{q-\alpha}\rightarrow 0, \quad \lambda\rightarrow\infty.
\end{align*}
Since $u_\lambda\rightarrow U$ in $C((0,\infty),L^1(\rr))$ and $\varphi\in C_c^\infty((0,\infty)\times \rr)$
we obtain
\[
  \int _0^\infty\int_{\R} |u_{\lambda}(t,x)-k|\partial_t \varphi dxdt\rightarrow   \int _0^\infty\int_{\R} |U(t,x)-k|\varphi dxdt.
\]
Observe that since $u_\lambda\rightarrow U$ in $C((0,\infty),L^1(\rr))$, function $U$ satisfies
\[
  \int_{\rr}U(t,x)dx=M.
\]

Moreover  $u_\lambda \rightarrow U$ a.e. in $(0,\infty)\times \rr$. This shows that the $L^\infty(\rr)$ bound in $u_\lambda$ transfers to $U$:
\[
  \| U(t)\|_{L^\infty(\rr)}\leq C(M) t^{-1/q}.
\]
This shows that $f(u_\lambda)\rightarrow f(U)$ in $C((0,\infty),L^1(\rr))$ and
\[
    \int _0^\infty\int_{\R} \sgn(u_{\lambda}-k)(f(u_{\lambda})-f(k))\partial _x\varphi dxdt\rightarrow
      \int _0^\infty\int_{\R} \sgn(U-k)(f(U)-f(k))\partial _x\varphi dxdt .
\]

In view of the fact that $I_1$ and $I_2$ tend to zero as $\lambda\rightarrow\infty$ we obtain that $U$ satisfies condition C1) in Definition \ref{entropysol}.

We now identify   the initial data taken by $U$ at $t=0$ by proving condition C2) in Definition \ref{entropysol}. Multiplying \eqref{Plambda} by $\psi\in C_c^2(\rr)$ the solution $u_\lambda$ satisfies
\[
  \int _{\rr} u_\lambda(t,x)\psi (x)dx-\int _{\rr}u_{\lambda}(0,x)\psi(x)dx=\int _0^t \int _{\rr} f(u_\lambda)\psi_x-\lambda^{q-\alpha}\int _0^t \int _{\rr} u_\lambda (-\Delta)^{\alpha/2}\psi.
\]
This implies that
\begin{align*}
\Big |\int _\rr  u_\lambda(t,x)\psi (x)dx&- \int _{\rr}u_0(x)\psi \big(\frac x\lambda\big)\Big |\leq \|\psi_x\|_{L^\infty(\rr)} \int _0^t \int_{\rr }u_\lambda^q dxds + t \lambda^{q-\alpha} M\|\psi\|_{H^\alpha(\rr)}   \\
  &\leq C(M) t^{1/q} \|\psi_x\|_{L^\infty(\rr)}+ t \lambda^{q-\alpha} M\|\psi\|_{H^\alpha(\rr)} .
\end{align*}
Passing to the limit $\lambda\rightarrow\infty$ and using that $q<\alpha$ we get that for any $\psi\in C_c^2(\rr)$ we have
\begin{equation}\label{initial.data.h2}
  \Big|\int _\rr  U(t,x)\psi (x)dx-M\psi(0)dx\Big |\leq t^{1/q}\|\psi\|_{H^2(\rr)} .
	\end{equation}
By density this estimate also holds for any $\psi\in H^2(\rr)$.

We now claim that
for any $\psi \in BC(\rr)$ the following holds
\begin{equation}
\label{initial.data}
  \lim _{t\rightarrow 0} \int _\rr  U(t,x)\psi (x)dx=M\psi(0) .
\end{equation}
This shows that $U$ is the unique entropy solution of system \eqref{limit.problem}. Since \eqref{limit.problem} has a unique solution, $U_M$, then the whole sequence $(u_\lambda)_{\lambda>0}$ converges to $U$ not only a subsequence.

We now prove  that  an approximation argument and the tail control of $u_\lambda$ (so of $U$)  give \eqref{initial.data} for any $\psi\in BC(\rr)$. Even this  procedure is standard, for completeness  we prefer to add it here. Let us choose a sequence of mollifiers $\{\rho_n\}_{n\geq 1}$ as in \cite[Ch.~4.4, p.~108]{MR2759829} and $\psi_n=\rho_n\ast \psi$. It follows that $\|\psi_n\|_{L^\infty(\rr)}\leq \|\psi\|_{L^\infty(\rr)}$ and $\psi_n\rightarrow\psi$ uniformly on compact sets of $\rr$ (cf. \cite[Prop.~4.2.1, Ch.~4, p.~108]{MR2759829}). Applying \eqref{initial.data.h2} to $\psi_n\in H^2(\rr)$ we obtain
\[
  \Big|\int _\rr  U(t,x)\psi_n (x)dx-M\psi_n(0)\Big |\leq t^{1/q}\|\psi_n\|_{H^2(\rr)} .
\]
We write
\begin{align*}
\label{}
  \int _\rr  U(t,x)\psi (x)dx-M\psi(0)=&\int _{|x|>2R}  U(t,x)(\psi(x)-\psi_n (x))dx\\
  & +\int _{|x|<2R}U(t,x)(\psi(x)-\psi_n(x))dx-M(\psi(0)-\psi_n(0))\\
 &+ \int _\rr  U(t,x)\psi_n (x)dx - M\psi_n(0)\\
 &=I+II+III.
\end{align*}
The uniform tail control in \eqref{tail} and the fact that for any $t>0$, $u_\lambda(t)\rightarrow U(t)$ in $L^1(\rr)$ give us, letting $\lambda\rightarrow\infty$, that $U$ satisfies something similar to \eqref{tail}:
\[
\int_{|x|>2 R} U(t,x) dx \leq \int_{|x|>R} u_0(x) dx + C( M)\frac{t^{1/q}}{R},\quad \forall t>0.
 \]
Hence
\begin{align*}
\label{}
 |I|= \Big| \int _{|x|>2R}  U(t,x)(\psi(x)-\psi_n (x))dx\Big|&\leq 2\|\psi\|_{L^\infty(\rr)}\int _{|x|>2R}U(t,x)dx\\
  &\leq 2\|\psi\|_{L^\infty(\rr)}\Big(\int_{|x|>R} u_0(x) dx + C( M)\frac{t^{1/q}}{R}\Big)< \epsilon,
\end{align*}
provided that $0<t<1$ and $R>R(\epsilon)$. Let us fix $R$ large enough. We analyze the second term  $II$. We have
\begin{align*}
\label{}
  \Big|\int _{|x|<2R}U(t,x)(\psi(x)-\psi_n(x))dx\Big|&\leq \|\psi-\psi_n\|_{L^\infty(|x|<2R)}\int _{\rr}U(t,x)dx\\
  &=M \|\psi-\psi_n\|_{L^\infty(|x|<2R)}.
\end{align*}
Thus, by using that $\psi_n\rightarrow \psi$ uniformly on compact sets of $\rr$ we get
\[
|II|\leq 2M \|\psi-\psi_n\|_{L^\infty(|x|<2R)}\leq \epsilon
\]
provided that $n$ is large enough. We now apply estimate \eqref{initial.data.h2} to $\psi_n$ to obtain
\[
| III |\leq t^{1/q}\|\psi_n\|_{H^2(\rr)}<\epsilon,
\]
provided $t$ is small enough. Hence $|I+II+III|\leq 3\epsilon$ for $t$ small enough, which finishes the proof of \eqref{initial.data}.

\medskip

\noindent \textbf{Step IV. Conclusion.} When $p=1$ we have proved that for any $t>0$, $u_\lambda(t)\rightarrow U_M(t)$ in $L^1(\rr)$. For  $p>1$ we use interpolation, the fact that $(u_\lambda(t))_{\lambda>0}$ is uniformly bounded in $L^{2p}(\rr)$ and that $U(t)\in L^{2p}(\rr)$. Indeed, we have
\[
  \|u_\lambda(t)-U_M(t)\|_{L^p(\rr)}\leq  \|u_\lambda(t)-U_M(t)\|_{L^1(\rr)}^{1/(2p-1)} (\|u_\lambda(t)\|_{L^{2p}(\rr)} +\|U_M(t)\|_{L^{2p}(\rr)})^{2(p-1)/(2p-1)} ,
\]
since $\frac{1}{p}=\frac{1-\theta}{1}+\frac{\theta}{2p}$ with $\theta=\frac{2(p-1)}{2p-1}.$
This proves the result for any $1\leq p<\infty$ and the proof is finished.
\end{proof}

\section{Appendix}

We give now  the proof of Lemma \ref{decay.nucleu}. We mention that these estimates were done in \cite{VazClassicalSolJEMS} for dimensions $N\ge 2$ and in the particular case $s=\alpha$ using some technical results of \cite{Pruitt}. We provide here the proof for all $s\in (0,2)$ and $\alpha\in (0,2)$ in the one-dimensional case. This requires a more careful proof since the results of \cite{Pruitt} allow only Bessel functions of positive index.

Using the homogeneity of the Fourier transform of $K_t^\alpha$ the proof is easily reduced to the case $t=1$. To simplify the presentation we will denote $K^\alpha$
the kernel $K_t^\alpha$ at the time $t=1$. In the first case we know (see \cite{BlumenthalGetoor}) that $K^\alpha$ satisfies
\[
  |K^\alpha(x)|\lesssim \frac{1}{|x|^{1+\alpha}}, \quad |x|>>1.
\]
The estimates on the $L^p(\rr)$ norm of $K^\alpha$ immediately follow.

We now want to estimate $(-\Delta)^\frac{s}{2} K^{\alpha}$. Using the Fourier transform we have
\[
(-\Delta)^\frac{s}{2} K^\alpha(x) =\frac 1{2\pi}\int_{-\infty}^{+\infty} e^{i x \xi}e^{-|\xi|^\alpha} |\xi|^s d\xi =\frac 1{\pi}\int_{0}^{+\infty} \cos(x \xi) e^{-|\xi|^\alpha} \xi^s  d\xi .
\]
and
\[
  (-\Delta)^\frac{s}{2} \partial_x K^\alpha(x) =\frac 1{2\pi}\int_{-\infty}^{+\infty} e^{i x \xi}e^{-|\xi|^\alpha} |\xi|^s (i\xi)d\xi =-\frac 1{\pi}\int_{0}^{+\infty} \sin(x \xi) e^{-|\xi|^\alpha} \xi^{s+1}  d\xi.
   \]

We consider the case when $x$ is positive and then
\[
  (-\Delta)^\frac{s}{2} K^\alpha(x) =\sqrt{\frac x{2\pi}} \int_{0}^{+\infty}  e^{-|\xi|^\alpha}   \xi^{s+{1/2}}J_{-1/2} (x\xi)d\xi
\]
and
\[
  (-\Delta)^\frac{s}{2} \partial_x K^\alpha(x) =-\sqrt{\frac x{2\pi}} \int_{0}^{+\infty}  e^{-|\xi|^\alpha}   \xi^{s+{3/2}}J_{1/2} (x\xi)d\xi,
\]
where $J_n$ is the Bessel function of first kind with index $n$.
We now use Lemma 1 in \cite{Pruitt} but we need to involve Bessel functions with positive index $J_\nu$, $\nu \geq 0$. In the second case applying this lemma we obtain that for $|x|$ large the following holds
\[
 | (-\Delta)^\frac{s}{2} \partial_x K^\alpha(x)|\lesssim \frac 1{|x|^{s+2}}.
\]
This shows that $(-\Delta)^\frac{s}{2} \partial_x K^\alpha$ belongs to $L^p(\rr)$ for any $1\leq p\leq \infty$.

In the first case we perform an integration by parts to obtain that
\[
    (-\Delta)^\frac{s}{2} K^\alpha(x)=-\frac{1}{2\pi x} \int _0^\infty e^{-|\xi|^\alpha}J_{1/2}(x\xi) \left( s \xi^{s-1/2} - \sigma \xi^{s+\alpha-1/2} \right) d\xi.
\]
Applying again Lemma 1 in \cite{Pruitt} we obtain that for $|x|$ large
\[
   | (-\Delta)^\frac{s}{2} K^\alpha(x)|\lesssim \frac 1{|x|^{s+1}}
\]
and then $(-\Delta)^\frac{s}{2} K^\alpha$ belongs to $L^p(\rr)$ for any $1\leq p\leq \infty$.

\bibliographystyle{siam}
\bibliography{biblioLIDS}

\begin{thebibliography}{10}

\bibitem{AlibaudEntropy}
{\sc N.~Alibaud}, {\em Entropy formulation for fractal conservation laws}, J.
  Evol. Equ., 7 (2007), pp.~145--175.

\bibitem{AlibaudAndreianov}
{\sc N.~Alibaud and B.~Andreianov}, {\em Non-uniqueness of weak solutions for
  the fractal {B}urgers equation}, Ann. Inst. H. Poincar\'e Anal. Non
  Lin\'eaire, 27 (2010), pp.~997--1016.

\bibitem{AlibaudImbertKarch}
{\sc N.~Alibaud, C.~Imbert, and G.~Karch}, {\em Asymptotic properties of
  entropy solutions to fractal {B}urgers equation}, SIAM J. Math. Anal., 42
  (2010), pp.~354--376.

\bibitem{Applebaum}
{\sc D.~Applebaum}, {\em L\'evy processes and stochastic calculus}, vol.~116 of
  Cambridge Studies in Advanced Mathematics, Cambridge University Press,
  Cambridge, second~ed., 2009.

\bibitem{Bertoin}
{\sc J.~Bertoin}, {\em L\'evy processes}, vol.~121 of Cambridge Tracts in
  Mathematics, Cambridge University Press, Cambridge, 1996.

\bibitem{BilerFunaki}
{\sc P.~Biler, T.~Funaki, and W.~A. Woyczynski}, {\em Fractal {B}urgers
  equations}, J. Differential Equations, 148 (1998), pp.~9--46.

\bibitem{BilerKarchWoyczAsymp}
{\sc P.~Biler, G.~Karch, and W.~A. Woyczy\'nski}, {\em Asymptotics for
  conservation laws involving {L}\'evy diffusion generators}, Studia Math., 148
  (2001), pp.~171--192.

\bibitem{BilerKarchWoycz2001}
{\sc P.~Biler, G.~Karch, and W.~A. Woyczy{\'n}ski}, {\em Critical nonlinearity
  exponent and self-similar asymptotics for {L}\'evy conservation laws}, Ann.
  Inst. H. Poincar\'e Anal. Non Lin\'eaire, 18 (2001), pp.~613--637.

\bibitem{BlumenthalGetoor}
{\sc R.~M. Blumenthal and R.~K. Getoor}, {\em Some theorems on stable
  processes}, Trans. Amer. Math. Soc., 95 (1960), pp.~263--273.

\bibitem{BSVfracHE}
{\sc M.~Bonforte, Y.~Sire, and J.~L. V{\'a}zquez}, {\em Optimal existence and
  uniqueness theory for the fractional heat equation}, Nonlinear Analysis:
  Theory, Methods and Applications, arXiv:1606.00873,  (2016).

\bibitem{Bouharguane2013}
{\sc A.~Bouharguane and R.~Carles}, {\em {Splitting methods for the nonlocal
  Fowler equation}}, Mathematics of Computation,  (2013), p.~1.

\bibitem{MR2759829}
{\sc H.~Brezis}, {\em Functional analysis, {S}obolev spaces and partial
  differential equations}, Universitext, Springer, New York, 2011.

\bibitem{Carpio1996}
{\sc A.~Carpio}, {\em Large time behaviour in convection-diffusion equations},
  Ann. Scuola Norm. Sup. Pisa Cl. Sci. (4), 23 (1996), pp.~551--574.

\bibitem{CazacuIgnatPazoto}
{\sc C.~Cazacu, L.~Ignat, and A.~Pazoto}, {\em On the asymptotic behavior of a
  subcritical convection-diffusion equation with nonlocal diffusion},
  Nonlinearity, 30 (2017), pp.~3126--3150.

\bibitem{MR1124294}
{\sc F.~M. Christ and M.~I. Weinstein}, {\em Dispersion of small amplitude
  solutions of the generalized {K}orteweg-de {V}ries equation}, J. Funct.
  Anal., 100 (1991), pp.~87--109.

\bibitem{CifaniJakobsenEntropySol}
{\sc S.~Cifani and E.~R. Jakobsen}, {\em Entropy solution theory for fractional
  degenerate convection-diffusion equations}, Ann. Inst. H. Poincar\'e Anal.
  Non Lin\'eaire, 28 (2011), pp.~413--441.

\bibitem{DroniouVanishing2003}
{\sc J.~Droniou}, {\em Vanishing non-local regularization of a scalar
  conservation law}, Electron. J. Differential Equations,  (2003), pp.~No. 117,
  20 pp. (electronic).

\bibitem{Droniou}
{\sc J.~Droniou, T.~Gallouet, and J.~Vovelle}, {\em Global solution and
  smoothing effect for a non-local regularization of a hyperbolic equation}, J.
  Evol. Equ., 3 (2003), pp.~499--521.
\newblock Dedicated to Philippe B{\'e}nilan.

\bibitem{DroniouImbertFractal}
{\sc J.~Droniou and C.~Imbert}, {\em Fractal first-order partial differential
  equations}, Archive for Rational Mechanics and Analysis, 182 (2006),
  pp.~299--331.

\bibitem{EndalJakobsen}
{\sc J.~Endal and E.~R. Jakobsen}, {\em {$L^1$} contraction for bounded
  (nonintegrable) solutions of degenerate parabolic equations}, SIAM J. Math.
  Anal., 46 (2014), pp.~3957--3982.

\bibitem{EVZArma}
{\sc M.~Escobedo, J.~L. V{\'a}zquez, and E.~Zuazua}, {\em Asymptotic behaviour
  and source-type solutions for a diffusion-convection equation}, Arch.
  Rational Mech. Anal., 124 (1993), pp.~43--65.

\bibitem{EVZIndiana}
{\sc M.~Escobedo, J.~L. V\'azquez, and E.~Zuazua}, {\em A diffusion-convection
  equation in several space dimensions}, Indiana Univ. Math. J., 42 (1993),
  pp.~1413--1440.

\bibitem{EZ}
{\sc M.~Escobedo and E.~Zuazua}, {\em Large time behavior for
  convection-diffusion equations in {${\bf R}^N$}}, J. Funct. Anal., 100
  (1991), pp.~119--161.

\bibitem{MR1878630}
{\sc A.~E. Gatto}, {\em Product rule and chain rule estimates for fractional
  derivatives on spaces that satisfy the doubling condition}, J. Funct. Anal.,
  188 (2002), pp.~27--37.

\bibitem{MR3328145}
{\sc L.~I. Ignat, T.~I. Ignat, and D.~Stancu-Dumitru}, {\em A compactness tool
  for the analysis of nonlocal evolution equations}, SIAM J. Math. Anal., 47
  (2015), pp.~1330--1354.

\bibitem{MR3190994}
{\sc L.~I. Ignat and A.~F. Pazoto}, {\em Large time behaviour for a nonlocal
  diffusion---convection equation related with gas dynamics}, Discrete Contin.
  Dyn. Syst., 34 (2014), pp.~3575--3589.

\bibitem{MR2356418}
{\sc L.~I. Ignat and J.~D. Rossi}, {\em A nonlocal convection-diffusion
  equation}, J. Funct. Anal., 251 (2007), pp.~399--437.

\bibitem{Imbert2005218}
{\sc C.~Imbert}, {\em A non-local regularization of first order
  hamilton–jacobi equations}, Journal of Differential Equations, 211 (2005),
  pp.~218 -- 246.

\bibitem{KamVaz88}
{\sc S.~Kamin and J.~L. V{\'a}zquez}, {\em Fundamental solutions and asymptotic
  behaviour for the {$p$}-{L}aplacian equation}, Rev. Mat. Iberoamericana, 4
  (1988), pp.~339--354.

\bibitem{KarchQualitProp2009}
{\sc G.~Karch}, {\em Nonlinear evolution equations with anomalous diffusion},
  in Qualitative properties of solutions to partial differential equations,
  vol.~5 of Jind\u rich Ne\u cas Cent. Math. Model. Lect. Notes, Matfyzpress,
  Prague, 2009, pp.~25--68.

\bibitem{KarchMiaoXu}
{\sc G.~Karch, C.~Miao, and X.~Xu}, {\em On convergence of solutions of fractal
  {B}urgers equation toward rarefaction waves}, SIAM J. Math. Anal., 39 (2008),
  pp.~1536--1549.

\bibitem{KarchPudelkoXu}
{\sc G.~Karch, A.~Pude\l~ko, and X.~Xu}, {\em Two-dimensional fractal {B}urgers
  equation with step-like initial conditions}, Math. Methods Appl. Sci., 38
  (2015), pp.~2830--2839.

\bibitem{LaurencotFast}
{\sc P.~Lauren\c{c}ot}, {\em Long-time behaviour for diffusion equations with
  fast convection}, Ann. Mat. Pura Appl. (4), 175 (1998), pp.~233--251.

\bibitem{MR2138795}
{\sc P.~Lauren\c{c}ot}, {\em Asymptotic self-similarity for a simplified model
  for radiating gases}, Asymptot. Anal., 42 (2005), pp.~251--262.

\bibitem{LaurencotSimondon}
{\sc P.~Lauren\c{c}ot and F.~Simondon}, {\em Long-time behaviour for porous
  medium equations with convection}, Proc. Roy. Soc. Edinburgh Sect. A, 128
  (1998), pp.~315--336.

\bibitem{MR735207}
{\sc T.-P. Liu and M.~Pierre}, {\em Source-solutions and asymptotic behavior in
  conservation laws}, J. Differential Equations, 51 (1984), pp.~419--441.

\bibitem{Pruitt}
{\sc W.~E. Pruitt and S.~J. Taylor}, {\em The potential kernel and hitting
  probabilities for the general stable process in {$R^{N}$}}, Trans. Amer.
  Math. Soc., 146 (1969), pp.~299--321.

\bibitem{Silvestre20112020}
{\sc L.~Silvestre}, {\em On the differentiability of the solution to the
  hamilton–jacobi equation with critical fractional diffusion}, Advances in
  Mathematics, 226 (2011), pp.~2020 -- 2039.

\bibitem{Simon}
{\sc J.~Simon}, {\em {Compact sets in the space $L\sp p(0,T;B)$.}}, Ann. Mat.
  Pura Appl., IV. Ser., 146 (1987), pp.~65--96.

\bibitem{MR2233925}
{\sc T.~Tao}, {\em Nonlinear dispersive equations}, vol.~106 of CBMS Regional
  Conference Series in Mathematics, Published for the Conference Board of the
  Mathematical Sciences, Washington, DC; by the American Mathematical Society,
  Providence, RI, 2006.
\newblock Local and global analysis.

\bibitem{Valdinoc}
{\sc E.~Valdinoci}, {\em From the long jump random walk to the fractional
  laplacian}, Bol. Soc. Esp. Mat. Apl., 49 (2009), pp.~33--44.

\bibitem{VazCIME}
{\sc J.~L. V\'azquez}, {\em The mathematical theories of diffusion. nonlinear
  and fractional diffusion}, To appear in Springer Lecture Notes in
  Mathematics\S, C.I.M.E. Subseries, 2017.

\bibitem{VazClassicalSolJEMS}
{\sc J.~L. V\'azquez, A.~de~Pablo, F.~Quir\'os, and A.~Rodr\'iguez}, {\em
  Classical solutions and higher regularity for nonlinear fractional diffusion
  equations}, J. Eur. Math. Soc, to appear, arXiv:1311.7427.

\bibitem{MR2318286}
{\sc M.~Visan}, {\em The defocusing energy-critical nonlinear {S}chr\"odinger
  equation in higher dimensions}, Duke Math. J., 138 (2007), pp.~281--374.

\bibitem{WoyczynskiLevyProc2001}
{\sc W.~A. Woyczy\'nski}, {\em L\'evy processes in the physical sciences}, in
  L\'evy processes, Birkh\"auser Boston, Boston, MA, 2001, pp.~241--266.

\end{thebibliography}

\begin{acknowledgements}\label{ackref}
This work has begun when D. Stan was a postdoctoral researcher at the Research Institute of the University of Bucharest (ICUB) during September-December 2015. She would like to acknowledge the warm hospitality of the Institute of Mathematics Simion Stoilow of the Romanian Academy during September-November 2015.

 L. Ignat acknowledges the Basque Center for Applied Mathematics for his visit in October 2016/November 2017.

 The authors are grateful to Professors Miguel Escobedo and Jerome Droniou for useful discussions. They also thank the anonymous referee for many useful comments which lead to a better comprehension of the paper.
\end{acknowledgements}


\affiliationone{Liviu Ignat\\
Institute of Mathematics Simion Stoilow \\
of the Romanian Academy, \\
Centre Francophone en Math\'ematique\\
Bucharest, Romania 
\email{liviu.ignat@imar.ro}}
\affiliationtwo{Diana Stan\\
Basque Center for Applied Mathematics,\\
Bilbao,  Spain. 
\email{dstan@bcamath.org}}

\end{document}